\documentclass[11pt]{amsart}
\usepackage{amsmath}
\usepackage{amssymb, latexsym, amsfonts, enumitem, tikz, fancyhdr, bm, mathrsfs}
\usepackage{xcolor}
\usepackage{todonotes}
\usepackage[normalem]{ulem}
\usepackage{fullpage}

\newtheorem{thm}{Theorem}[section]
\newtheorem{prop}[thm]{Proposition}
\newtheorem{lem}[thm]{Lemma}
\newtheorem{defn}[thm]{Definition}

\newtheorem{rem}[thm]{Remark}
\newtheorem{cor}[thm]{Corollary}

\renewcommand{\leq}{\leqslant}
\renewcommand{\geq}{\geqslant}

\newcommand{\diag}{\textup{diag}}

\numberwithin{equation}{section}

\begin{document}

\title{Extreme points, positive Grothendieck constants\\ and tensor product norms}
\author[Gupta]{Rajeev Gupta}
\author[Mal]{Arpita Mal}
\author[Misra]{Gadadhar Misra}
\author[Ray]{Samya Kumar Ray}

\address[Gupta]{School of Mathematics and Computer Science, Indian Institute of Technology Goa, Goa - 403401}
\email[Gupta]{\text rajeev@iitgoa.ac.in}

\address[Ray]{The Institute of Mathematical Sciences, 4th Cross Street, CIT Campus, Tharamani, Chennai, Tamil Nadu 600113, India, and Homi Bhabha National Institute, Training School Complex, Anushakti Nagar, Mumbai 400094, India} 
\email[Ray]{\text samya@imsc.res.in}

\address[Mal]{Dhirubhai Ambani University, Gandhinagar, India}
\email[Mal]{\text arpita\_mal@dau.ac.in}

\address[Misra]{Indian Statistical Institute, Bangalore, India and Indian Institute of Technology, Gandhinagar, India}
\email[Misra]{\text gm@isibang.ac.in}

\begin{abstract}
We study several interrelated problems arising from the interplay between extreme point theory, Grothendieck-type inequalities, and tensor product norms. We develop a general framework for characterizing the extreme points of the set of positive contractions $\mathcal{A}_{X\to Y}$ between finite-dimensional Banach spaces, with explicit results for $X=\ell_1^n$, $Y=\ell_\infty^n$ and vice versa. These characterizations are applied to evaluate several constants exactly. We show that the positive Grothendieck constant $K_G^{+,\mathbb{R}}(3)$ equals $9/8$ and that the smallest constant $\rho^{+}(X)$ for which $\|A\|_\pi \leqslant \rho^{+}(X)\|A\|_\epsilon$ holds for all $A \geqslant 0$ equals $5/4$ when $X=\ell^3_\infty(\mathbb{R})$.
We also prove that $\rho^+(X)=1$ when $X=\ell_\infty^n(\mathbb{C})$ and $n\leqslant 3$. Finally, we prove that $\rho^+(X) = 1$ for every 2-dimensional subspace $X$ of $\ell^3_\infty(\mathbb{C})$; since this is stronger than the 2-summing property, 
it recovers Proposition~4.4 of \cite{AFJS95}.
\end{abstract}

\maketitle
\setcounter{section}{1}
\section*{Introduction}
Let $\Omega \subset \mathbb C^m$ be a bounded domain and let 
$H^\infty(\Omega)$ denote the algebra of bounded holomorphic functions on $\Omega$. 
For an $m$-tuple $V=(V_1,\dots,V_m)$ with $V_j \in M_{n_1,n_2}(\mathbb C)$,
the associated \emph{Parrott homomorphism} is defined by
\[
\rho_V(f)=
\begin{pmatrix}
f(w)I_{n_1} &
\displaystyle \sum_{j=1}^m
\frac{\partial f}{\partial z_j}(w)V_j \\[1ex]
0 & f(w)I_{n_2}
\end{pmatrix},
\qquad f\in H^\infty(\Omega),\; w\in\Omega .
\]

The $m$-tuple $V=(V_1,\dots,V_m)$ also defines a linear map $L_V:E^* \to M_{n_1,n_2}(\mathbb{C}), L_V(z_1,\dots,z_m) = \sum_{j=1}^m z_jV_j$, where $M_{n_1,n_2}(\mathbb{C})$ denotes the space of $n_1\times n_2$ complex matrices equipped with the usual operator norm.  Building on earlier results from \cite{MS, GM}, it was established in \cite{vip}  that the contractivity (respectively, complete contractivity) of the homomorphism $\rho_V$ is equivalent to the contractivity (respectively, complete contractivity) of the linear map $L_V$. This reveals a deep interplay between Arveson's boundary normal dilations over $\Omega$ and the operator space structure of Banach spaces. 
Moreover, as observed in \cite{BM}, many problems involving the homomorphisms $\rho_V$ can be reduced to equivalent questions in Banach space geometry. This led to the introduction of Banach spaces with Property $P$ and Property $Q$ in \cite{BM}. Let $X$ be a finite dimensional Banach space. 
Property~P asks that $\langle A, B\rangle \leqslant\|A\|_\epsilon\|B\|_\epsilon$ for all non-negative definite $A \in X^* \otimes X^*$ and all non-negative definite $B \in X \otimes X$. The quantitative form of Property P is the constant $\gamma^+(X)$ of Definition \ref{defn:propP}; for $X=\ell_\infty^n$ it equals the positive Grothendieck constant $K_G^{+,\mathbb{R}}(n)$. Similarly, the constant $\rho^+(X)$ is the quantitative form of Property~Q which asks when the injective and projective tensor norms agree on the non-negative definite elements of $X \otimes X$. 
 
This paper studies several interrelated problems arising from the interplay between extreme point theory, Grothendieck-type inequalities, and the structure of operator spaces. The unifying theme is the characterization of extreme points of various convex sets of matrices --- positive contractions, correlation matrices, and tuples of contractions subject to coupled norm constraints --- and the application of these characterizations to compute or estimate constants of fundamental importance in operator space theory.

In Section~\ref{sec:AXY} we develop a general framework for characterizing extreme points of the set $\mathcal{A}^{(n)}_{X\to Y}(\mathbb F)$ 
of non-negative definite  matrices $T$ satisfying $\|T\|_{X\to Y}\leq 1$, where $X$ and $Y$ are $n$-dimensional Banach spaces. The general theory is then specialized to the cases $X=\ell_1^n$, $Y=\ell_\infty^n$, where we obtain general results for the complex and the real settings, see Theorem \ref{th-complete} and Theorem \ref{th-realcomplete}, respectively. Describing the extreme points of $\mathcal{A}^{(n)}_{X\to Y}(\mathbb F)$, where $X=\ell_\infty^n$ and $Y=\ell_1^n$, is more difficult. Irrespective of the ground field, Theorem~\ref{th-2dim} describes these for $n=2$. Over the real field, Theorem~\ref{thm-rank2-inftyto1} gives a necessary condition 
for extremality when $n=3$, and Theorem~\ref{thm-normattain} provides another.

In Section~\ref{sec:KG}, we apply the extreme point analysis of $\mathcal{A}_{1\to\infty}$ and $\mathcal{A}_{\infty\to 1}$ over the real field to show that the positive Grothendieck constant $K_G^{+,\mathbb{R}}(3)=9/8$ (Theorem \ref{thm-KG}). The extreme point analysis of Section~\ref{sec:KG} also underlies the computation of $\rho^+(\ell_\infty^3(\mathbb{R}))=5/4$ (Theorem \ref{thm-Q}) in Section \ref{sec:propertyQ-constant}. 

Section~\ref{sec:propertyQ} is devoted to the study of Property~Q. We prove that $\ell_{\infty}^n(\mathbb{C})$ has Property~Q 
for $n \leqslant 3$. In \cite{BM}, Property P is shown to be equivalent to the 2-summing property. If $X$ has Property~Q, then 
for any non-negative definite $A \in X^* \otimes X^*$ and $B \in X \otimes X$, the duality between the injective and projective 
norms gives $\langle A, B\rangle \leqslant\|A\|_\epsilon\|B\|_\pi= \|A\|_\epsilon\|B\|_\epsilon$, so Property~Q implies 
Property~P.  Since \cite[Example 2.3]{AFJS95} proves that the $2$-summing property fails for $\ell_\infty^n(\mathbb{C})$ 
when $n>3$, it follows that $\ell_{\infty}^n(\mathbb{C})$ has Property~Q if and only if $n \leqslant 3$.

Our proof also resolves a gap in the argument of \cite[Fact~7]{BM} asserting that $\ell_\infty^3(\mathbb C)$ has Property~Q. The proof in \cite{BM} proceeds by considering extreme points of the set of non-negative matrices of injective norm at most $1$ and assumes that all the diagonal entries of every such extreme point are  equal to $1$. However, Corollary~\ref{cor-complex-n3} shows that this assumption is false. Consequently, the argument given in \cite{BM} is incomplete. The difficulty can be overcome by passing between the set of non-negative matrices of injective norm at most $1$ and the set of correlation matrices. Using this correspondence, together with the rank constraint for extreme correlation matrices established in \cite[Theorem~3]{CV}, we obtain a complete proof of Property~Q for $\ell_\infty^3(\mathbb C)$. This approach avoids the auxiliary calculus lemma \cite[Lemma~2.2]{BM} used in the earlier proof.

In Subsection~\ref{sec:propertyQ-normA}, we prove that $\left(\mathbb{C}^2,\|\cdot\|_{\boldsymbol A}\right)$, where $\left\|\left(z_1, z_2\right)\right\|_{\boldsymbol A}:= \left\|z_1 A_1+z_2 A_2\right\|$ for a fixed but arbitrary pair of $3 \times 3$ diagonal matrices $A_1, A_2$, has Property~$\mathrm{Q}$. It then follows that every $2$-dimensional subspace of $\ell_{\infty}^3$ has Property $Q$ (Corollary \ref{cor:2dimPropQ}). Since Property~Q is stronger than Property~P, or equivalently the $2$-summing property, this recovers the result of \cite[Proposition~4.4]{AFJS95} asserting that every $2$-dimensional subspace of $\ell_\infty^3$ has the $2$-summing property.

\section{Extreme points of the set $\mathcal{A}_{X\to Y}$}\label{sec:AXY}

Let $X$ be a finite-dimensional vector space and $C\subseteq X$ be a convex subset. Recall that an open line segment is a set of the form
\[
(x, y) := \{ t x + (1 - t)y : 0 < t < 1 \},
\]
and the segment is said to be proper if $x \neq y$.
\begin{defn}
Let $C$ be a convex subset of a vector space $X$. A point $u \in C$ is called an {extreme point} of $C$ if there is no proper open line segment that contains $u$ and lies entirely in $C$. We let $\operatorname{Ext}(C)$ denote the set of all extreme points of $C$.
\end{defn}
For any two Banach spaces $X$ and $Y$ we denote $B(X,Y)$ to be all bounded linear maps from $X$ to $Y.$ Suppose $X, Y$ are $n$-dimensional Banach spaces over $\mathbb{F}$, where $\mathbb{F}=\mathbb{R}$ or $\mathbb{C}$. By fixing some basis $\mathcal{B}_X$ and $\mathcal{B}_Y$ for $X$ and $Y$ respectively we may identify $X$ and $Y$ with $(\mathbb F^n,\|\cdot\|_X)$ and $(\mathbb{F}^n,\|\cdot\|_Y)$ respectively. 
Then, we identify a linear transformation $T:X\to Y$ with an element in $ M_{n}(\mathbb{F})$ and vice versa. Let us define
\[
\mathcal{A}_{X\to Y}^{(n)}(\mathbb F):=\{T\in M_{n}(\mathbb{F}): T\geq 0,\, \|T\|_{(\mathbb F^n,\|\cdot\|_X)\to (\mathbb F^n,\|\cdot\|_Y)}\leq 1\}.
\] 
Whenever, it is clear from the context we write  $\|T\|$ in place of  $\|T\|_{(\mathbb F^n,\|\cdot\|_X)\to (\mathbb F^n,\|\cdot\|_Y)}.$ 
Moreover, we will drop the argument $\mathbb F$ from $\mathcal{A}_{X\to Y}^{(n)}(\mathbb F)$ whenever the underlying field is clear from the context. 
Suppose $T\in \mathcal{A}_{X\to Y}^{(n)}(\mathbb F)$ and $\operatorname{rank}(T)=r$. Let $\lambda_1\geq \lambda_2\geq\cdots\geq \lambda_r$ be the non-zero eigenvalues of $T$. Set
\begin{equation}\label{notationforGamma}
\tilde{\Gamma}=\diag(\lambda_1,\ldots,\lambda_r) \quad\text{and}\quad \Gamma= \begin{bmatrix}
		\tilde{\Gamma} & 0 \\
		0& 0
	\end{bmatrix}\in M_n(\mathbb{R}),
\end{equation}
  where $\diag(\lambda_1,\ldots,\lambda_r)$ denotes the $r\times r$ diagonal matrix with diagonal entries $\lambda_1,\ldots,\lambda_r.$
Then there exists a unitary matrix $A\in M_{n}(\mathbb{F})$ such that $T=A\Gamma A^*$. Write $T=(t_{ij})_{i,j=1}^n$ and $A=(a_{ij})_{i,j=1}^n$.

\begin{lem}\label{lem-g01}
Suppose $T\in\mathcal{A}_{X\to Y}^{(n)}(\mathbb{F})$ and $\operatorname{rank}(T)=r$. Then the following are equivalent.

\noindent (i) $T$ is not an extreme contraction.

\noindent (ii) There exists a non-zero self-adjoint matrix $\tilde{C}\in M_{r}(\mathbb{F})$ such that 
\[
\tilde{\Gamma}\pm \tilde{C}\geq 0 \quad\text{and}\quad\|T\pm C\|_{X\to Y}\leq 1,
\]
where $C=A\begin{bmatrix} \tilde{C} & 0 \\ 0& 0 \end{bmatrix}A^*$.
\end{lem}
\begin{proof}
$(i)\Rightarrow (ii)$. 
Recall that a point $u$ in a convex set $C$ is not an extreme point of $C$ if and only if there exist a nonzero vector $v$ and some $t_0>0$ such that
\[
u \pm t v \in C \quad \text{for all} \quad 0<t \leq t_0.
\]
Suppose $T$ is not an extreme contraction. Then there exists a non-zero matrix $C\in M_{n}(\mathbb{F})$ such that
\[
T=\frac{1}{2}(T+C)+\frac{1}{2}(T-C),
\]
where $T\pm C\geq 0$ and $\|T\pm C\|_{X\to Y} \leq 1$. Since $T\pm C\geq 0$, conjugating by $A^*$ gives $\Gamma \pm A^*CA \geq 0$. In particular, the diagonal entries satisfy $(\Gamma \pm A^*CA)_{ii} \geq 0$ for all $1\leq i\leq n$. For $r+1\leq i\leq n$, the diagonal entry $\Gamma_{ii}=0$, so we must have $\pm (A^*CA)_{ii} \geq 0$, which forces  $(\Gamma \pm A^*CA)_{ii} = 0$ for $r+1\leq i\leq n.$ 
Moreover, as $\Gamma \pm A^*CA$ is non-negative definite, all entries in the corresponding row and column must vanish: $(A^*CA)_{ij}= (A^*CA)_{ji} = 0$ for $r+1\leq i\leq n$ and $1\leq j\leq n$. Hence we have that $A^*CA=\begin{bmatrix} \tilde{C} & 0 \\ 0& 0 \end{bmatrix}$ and equivalently $C=A\begin{bmatrix} \tilde{C} & 0 \\ 0& 0 \end{bmatrix}A^*$ for some $\tilde{C}\in M_{r}(\mathbb{F})$. Clearly, $C\neq 0$ implies $\tilde{C}\neq 0$. Moreover, $\Gamma \pm A^*CA\geq 0$ implies $\tilde{\Gamma}\pm \tilde{C}\geq 0$. Thus $\tilde{C}=\frac{1}{2}(\tilde{\Gamma}+\tilde{C})-\frac{1}{2}(\tilde{\Gamma}-\tilde{C})$. Hence $\tilde{C}$ is self-adjoint.

$(ii)\Rightarrow (i)$. Suppose there exists $\tilde{C}\in M_{r}(\mathbb{F})$ satisfying the conditions of $(ii)$. Note that $T=\frac{1}{2}(T+C)+\frac{1}{2}(T-C)$. Since $\tilde{\Gamma}\pm \tilde{C}\geq 0,$ it follows that $T\pm C\geq 0.$  Therefore $T$ is not an extreme contraction. This completes the proof of the lemma.
\end{proof}

\begin{cor}\label{cor-g01}
Suppose $T\in\mathcal{A}^{(n)}_{X\to Y}(\mathbb F)$ and $\operatorname{rank}(T)=r$. Then the following are equivalent.

\noindent (i) $T$ is an extreme contraction.

\noindent (ii) The set
\begin{equation*}\label{eq-g1}
\Big\{\tilde{C}\in M_r(\mathbb{F}):\tilde{C}^*=\tilde{C},\, \tilde{\Gamma}\pm \tilde{C}\geq 0,\, \|T\pm C\|_{X\to Y} \leq 1,\, C=A\begin{bmatrix} \tilde{C} & 0 \\ 0& 0 \end{bmatrix}A^*\Big\}
\end{equation*}
is the singleton set $\{0\}$.
\end{cor}

\begin{prop}\label{prop-g1}
Let $T\in \mathcal{A}_{X\to Y}^{(n)}(\mathbb F)$ be such that $0<\|T\|<1$. Then $T$ is not an extreme contraction.
\end{prop}
\begin{proof}
 Clearly, $\big(1-\frac{1}{\|T\|}\big)<0$. Therefore, there exists a scalar $c>0$ such that $1\pm c>0$ and $\big(1-\frac{1}{\|T\|}\big)< \pm c$.  Thus,
\[
1-\frac{1}{\|T\|}<\pm {c}\quad \Rightarrow \quad 1\pm{c}<\frac{1}{\|T\|}\quad \Rightarrow \quad 0< ({1\pm c})\|T\|<1.
\]
Consider the matrix $C={c}T$. Then $C$ is non-zero and
\[
\|T\pm C\|_{X \to Y}=\|(1\pm{c})T\|_{X\to Y}= ({1\pm c})\|T\|_{X\to Y}<1.
\]
Moreover, $T\pm C=({1\pm c}) T\geq 0$. Since $T=\frac{1}{2}(T+C)+\frac{1}{2}(T-C)$, it follows that $T$ is not an extreme contraction.
\end{proof}

\begin{thm}\label{th-grank1}
Suppose $T\in \mathcal{A}_{X\to Y}^{(n)}(\mathbb{F})$ and $\operatorname{rank}(T)=1$. Then the following are equivalent.
\begin{itemize}
    \item[(i)]  $T$ is an extreme contraction.
\item[(ii)] $\|T\|_{X\to Y}=1$.
\end{itemize}
\end{thm}
\begin{proof}
$(i)\Rightarrow (ii)$ follows from Proposition~\ref{prop-g1}.

$(ii)\Rightarrow (i)$. Suppose, if possible, that $T$ is not an extreme contraction. Then by Lemma~\ref{lem-g01}, since $\operatorname{rank}(T)=1$, there exists a non-zero real scalar $c$ such that $\lambda_1\pm c\geq 0$ and $\|T\pm C\|_{X\to Y}\leq 1$, where
$C=A\begin{bmatrix} c & 0 \\ 0& 0 \end{bmatrix}A^*$ with notations from \eqref{notationforGamma}.
Now,
\[
\frac{\lambda_1\pm c}{\lambda_1}\|T\|_{X\to Y}=\Big\|\frac{\lambda_1\pm c}{\lambda_1}T\Big\|_{X\to Y}=\|T\pm C\|_{X\to Y}\leq 1,
\]
which shows that $\lambda_1\pm c\leq \lambda_1$, that is, $c=0$. This contradiction proves that $T$ must be an extreme contraction.
\end{proof}

\subsection{The case $X=\ell_1^n$, $Y=\ell_\infty^n$}

In this section, all the matrix computations are with respect to the canonical basis of $\mathbb{F}^n.$ 
We now characterize the extreme contractions of $\mathcal{A}_{\ell_1^n\to \ell_\infty^n}^{(n)}(\mathbb F)$, denoted by $\mathcal{A}_{1\to \infty}^{(n)}({\mathbb{F}})$ for brevity. If $T=(t_{ij})\in \mathcal{A}_{1\to\infty}^{(n)}({\mathbb{F}})$, then, using Cauchy-Schwarz inequality,  $\|T\|_{1\to\infty}=\max\{t_{ii}: 1\leq i\leq n\}$. Therefore, by Corollary~\ref{cor-g01}, if $\operatorname{rank}(T)=r$, then $T$ is an extreme contraction of $\mathcal{A}_{1\to\infty}^{(n)}({\mathbb{F}})$ if and only if the set
\begin{equation}\label{eq-n1}
\Big\{\tilde{C}\in M_r(\mathbb{F}):\tilde{C}^*=\tilde{C},\, \tilde{\Gamma}\pm \tilde{C}\geq 0,\, 0\leq (T\pm C)_{ii} \leq 1\;\forall\;1\leq i\leq n \Big\},
\end{equation}
where $C=A\begin{bmatrix} \tilde{C} & 0 \\ 0& 0 \end{bmatrix}A^*$, 
is the singleton set $\{0\}$.

\begin{cor}
Suppose $T\in \mathcal{A}_{1\to\infty}^{(n)}({\mathbb{F}})$ is such that $|t_{ij}|=1$ for all $1\leq i,j\leq n$. Then $T$ is an extreme contraction.
\end{cor}
\begin{proof}
Since $|t_{ij}| = 1$ for all $1 \leq i,j \leq n$, every $2 \times 2$ principal minor of $T$ vanishes. Let $v_1,\dots,v_n$ be such that $T = (\langle v_i, v_j \rangle)_{i,j=1}^n.$
By \cite[Theorem~7.2.10]{HJ} the vanishing of all $2 \times 2$ principal minors implies that the vectors $v_1,\dots,v_n$ are pairwise linearly dependent, and hence
$\dim \operatorname{span}\{v_i : 1 \leq i \leq n\} = 1.$
Therefore, again by \cite[Theorem~7.2.10]{HJ} $\operatorname{rank}(T) = 1$. The result now follows from Theorem~\ref{th-grank1}.
\end{proof}

We now focus our attention for $\mathbb F=\mathbb C.$ Suppose $\tilde{C}=(c_{ij})$ is an element of the set given in \eqref{eq-n1}. Without loss of generality, suppose that $t_{ii}=1$ for $1\leq i\leq m$ and $t_{ii}<1$ for $m+1\leq i\leq n$. Clearly, $(C)_{ii}=0$ for all $1\leq i\leq m$. Or, equivalently 
\begin{equation}\label{eq-02}
\sum_{p=1}^r |a_{ip}|^2 c_{pp}+\sum_{1\leq p<l\leq r} \big[2\Re\{a_{ip}\overline{a}_{il}\}\Re\{c_{pl}\}-2\Im\{a_{ip}\overline{a}_{il}\}\Im\{c_{pl}\}\big]=0
\end{equation}
for all $1\leq i\leq m$. This is a system of $m$ linear equations with $r^2$ real variables and real coefficients. It defines a linear operator from a real vector space $V$ of dimension $r^2$ to $\mathbb{R}^{m}$, where
\[
V=\big\{C'=(\ldots,c_{jj},\ldots,\Re(c_{pl}), \Im(c_{pl}),\ldots): c_{jj},\Re (c_{pl}),\Im(c_{pl})\in \mathbb{R},\; 1\leq j\leq r,\, 1\leq p<l\leq r\big\}.
\]
Define a linear operator $L:V\to\mathbb{R}^{m}$ by
\[
L(C')_{s} =\sum_{p=1}^r |a_{sp}|^2 c_{pp}+\sum_{1\leq p<l\leq r} \big[2\Re\{a_{sp}\overline{a}_{sl}\}\Re\{c_{pl}\}-2\Im\{a_{sp}\overline{a}_{sl}\}\Im\{c_{pl}\}\big]
\]
for $1\leq s\leq m$. If $\ker(L)=\{0\}$, then $T$ is an extreme contraction.

\begin{rem}\label{remark-real}
If we consider real scalars instead of complex scalars, then the system of equations \eqref{eq-02} involves $m$ equations and $\frac{r^2+r}{2}$ variables. The real vector space $V$ has dimension $\frac{r^2+r}{2}$, and the operator $L:V\to\mathbb{R}^{m}$ is defined by
\[
L(C')_{s} =\sum_{p=1}^r |a_{sp}|^2 c_{pp}+\sum_{1\leq p<l\leq r} 2a_{sp}{a}_{sl}\, c_{pl}
\]
for $1\leq s\leq m$. As in the complex case, $\ker(L)=\{0\}$ implies that $T$ is an extreme contraction.
\end{rem}

\begin{thm}\label{th-complete}
 Suppose $T\in \mathcal{A}_{1\to\infty}^{(n)}({\mathbb{C}})$ and $\operatorname{rank}(T)=r$. Suppose that $t_{ii}=1$ for $1\leq i\leq m$ and $t_{ii}<1$ for $m+1\leq i\leq n$. Then the following are equivalent.

\noindent (i) $T$ is an extreme contraction.

\noindent (ii) $\operatorname{rank}(L)=r^2$, where $L$ is as above.

\noindent (iii) $\operatorname{rank}(\tilde{A})=r^2$, where $\tilde{A}=[P\ Q\ R]\in M_{m\times r^2}(\mathbb{R})$ with $P\in M_{m\times r}$ and $Q, R\in M_{m\times \frac{r(r-1)}{2}}$ with $s$-th column of $P$ is $(|a_{1s}|^2, \ldots, |a_{ms}|^2)^T$ for each $1\leq s \leq r$ whereas $(p,l)$-th column of $Q$ and $R$ are $(2\Re\{a_{1p}\overline{a}_{1l}\},\ldots, 2\Re\{a_{mp}\overline{a}_{ml}\})^T$ and $(-2\Im\{a_{1p}\overline{a}_{1l},\ldots, -2\Im\{a_{mp}\overline{a}_{ml}\})^T$ respectively for each $1\leq p<l\leq r.$ Thus, $\tilde{A}$ is of the form
\[
\tilde{A}=\begin{bmatrix}
\ldots&	|a_{1s}|^2&\ldots &2\Re\{a_{1p}\overline{a}_{1l}\}&-2\Im\{a_{1p}\overline{a}_{1l}\}&\ldots\\
\ldots &	|a_{2s}|^2 &\ldots&2\Re\{a_{2p}\overline{a}_{2l}\}&-2\Im\{a_{2p}\overline{a}_{2l}\}&\ldots\\
\vdots&	\vdots & \vdots &\vdots & \vdots &\vdots\\
\ldots&	|a_{ms}|^2 &\ldots&2\Re\{a_{mp}\overline{a}_{ml}\}&-2\Im\{a_{mp}\overline{a}_{ml}\}&\ldots
\end{bmatrix}.
\]
\end{thm}
\begin{proof}
$(ii)\Rightarrow (i)$. Since $\operatorname{rank}(L)=r^2$ implies $\ker{L}=\{0\}$, the previous discussion gives that $T$ is an extreme contraction.

$(i)\Rightarrow (ii)$. Suppose $T$ is an extreme contraction and, 
seeking a contradiction, suppose $\operatorname{rank}(L)\neq r^2$. 
Then $\ker(L)$ contains a non-zero element, which corresponds to a 
non-zero self-adjoint matrix $\tilde{C}\in M_r(\mathbb{C})$ such that 
$C=A\begin{bmatrix} \tilde{C} & 0 \\ 0 & 0 \end{bmatrix}A^*$ 
satisfies $(C)_{ii}=0$ for all $1\leq i\leq m$.

Choose $\delta>0$ (to be adjusted) so that the diagonal entries of 
$T\pm\delta C$ do not exceed~$1$: if $m=n$ this holds for all 
$\delta$ since $(C)_{ii}=0$ for every~$i$; if $m<n$ it holds for 
$\delta$ small enough by continuity, since $t_{ii}<1$ for $m+1\leq 
i\leq n$.

It remains to ensure $T\pm\delta C\geq 0$, which is equivalent to 
$\tilde{\Gamma}\pm\delta\tilde{C}\geq 0$. Since diagonal elements of $\tilde{\Gamma}$ are all strictly positive, we can choose $\delta>0$ such that $T\pm\delta C\geq 0$ and $\|T\pm\delta 
C\|_{1\to\infty}\leq 1$, so $T$ is not extreme by 
Lemma~\ref{lem-g01} --- a contradiction.

$(ii)\Leftrightarrow (iii)$. This follows since the matrix associated with the operator $L$ is $\tilde{A}$.
\end{proof}

\begin{cor}\label{cor-02}
Suppose $T\in \mathcal{A}_{1\to\infty}^{(n)}({\mathbb{C}})$ and $\operatorname{rank}(T)=r$. If $T$ is an extreme contraction, then $r^2\leq m$, where $m$ is the number of diagonal entries of $T$ equal to $1$.
\end{cor}
\begin{proof}
From Theorem~\ref{th-complete}, $T$ is an extreme contraction if and only if $\operatorname{rank}(\tilde{A})=r^2$, where $\tilde{A}\in M_{m,r^2}(\mathbb{R})$. Since $\operatorname{rank}(\tilde{A})\leq m$, if $T$ is an extreme contraction then $r^2\leq m$.
\end{proof}

Combining Theorem~\ref{th-grank1} and Corollary~\ref{cor-02} gives the following explicit characterization for small $n$.

\begin{cor}\label{cor-complex-n3}
 Suppose $T\in \mathcal{A}_{1\to\infty}^{(3)}({\mathbb{C}})$. Then $T$ is an extreme contraction if and only if either (i) $T=0$, or (ii) $\operatorname{rank}(T)=1$ and at least one diagonal entry of $T$ equals $1$.
\end{cor}

\begin{cor}\label{cor-03}
 $T\in \mathcal{A}_{1\to\infty}^{(4)}({\mathbb{C}})$. Then $T$ is an extreme contraction if and only if either (i) $T=0$, or (ii) $\operatorname{rank}(T)=1$ and at least one diagonal entry of $T$ equals $1$, or (iii) $\operatorname{rank}(T)=2$ and $T$ is an extreme correlation matrix.
\end{cor}
\begin{proof}
If any of (i), (ii), or (iii) holds, then $T$ is an extreme contraction. In particular, since the set of correlation matrices is a face of $\mathcal{A}_{1\to\infty}^{(n)}({\mathbb{C}})$, an extreme correlation matrix must be an extreme contraction. Conversely, if $T$ is an extreme contraction with $\operatorname{rank}(T)=r$, then Corollary~\ref{cor-02} gives $r\leq 2$. If $r=2$, all diagonal entries of $T$ are $1$, so $T$ is a correlation matrix; moreover, it must be an extreme correlation matrix. If $r=1$, then (ii) follows from Theorem~\ref{th-grank1}.
\end{proof}

For real scalars, using Remark~\ref{remark-real} and proceeding as in Theorem~\ref{th-complete}, we obtain the following.

\begin{thm}\label{th-realcomplete}
Suppose $T\in \mathcal{A}_{1\to\infty}^{(n)}({\mathbb{R}})$ and $\operatorname{rank}(T)=r$. Suppose that $t_{ii}=1$ for $1\leq i\leq m$ and $t_{ii}<1$ for $m+1\leq i\leq n$. Then the following are equivalent.

\noindent (i) $T$ is an extreme contraction.

\noindent (ii) $\operatorname{rank}(L)=\frac{r^2+r}{2}$, where $L$ is as in Remark~\ref{remark-real}.

\noindent (iii) $\operatorname{rank}(\tilde{A})=\frac{r^2+r}{2}$, where $\tilde{A}=[P\ \ Q]\in M_{m\times \frac{r^2+r}{2}}(\mathbb{R})$ with $P\in M_{m\times r}$ and $Q\in M_{m\times \frac{r(r-1)}{2}}$ with $s$-th column of $P$ is $(|a_{1s}|^2, \ldots, |a_{ms}|^2)^T$ for each $1\leq s \leq r$ whereas $(p,l)$-th column of $Q$ is $(2a_{1p}{a}_{1l},\ldots, 2a_{mp}{a}_{ml})^T$ for each $1\leq p<l\leq r.$ Thus, $\tilde{A}$ is of the form
\[
\tilde{A}=\begin{bmatrix}
\ldots&	|a_{1s}|^2&\ldots &2a_{1p}{a}_{1l}&\ldots\\
\ldots &	|a_{2s}|^2 &\ldots&2a_{2p}{a}_{2l}&\ldots\\
\vdots&	\vdots & \vdots &\vdots &\vdots\\
\ldots&	|a_{ms}|^2 &\ldots&2a_{mp}{a}_{ml}&\ldots
\end{bmatrix}.
\]
\end{thm}

\begin{cor}\label{cor-02real}
Suppose $T\in \mathcal{A}_{1\to\infty}^{(n)}({\mathbb{R}})$ and $\operatorname{rank}(T)=r$. If $T$ is an extreme contraction, then $\frac{r^2+r}{2}\leq m$, where $m$ is the number of diagonal entries of $T$ equal to $1$.
\end{cor}

We recall the following result of Li and Tam.

\begin{thm}[{\cite[Corollary~5]{LT}}]\label{th-LT}
A $3\times 3$ real symmetric correlation matrix of rank two is extreme if and only if its off-diagonal entries all have absolute values less than $1$.
\end{thm}

Using Corollary~\ref{cor-02real} and Theorem~\ref{th-LT}, and proceeding as in Corollary~\ref{cor-03}, we obtain the following.

\begin{cor}\label{cor-realn=3}
Suppose $T\in \mathcal{A}_{1\to\infty}^{(3)}({\mathbb{R}})$. Then $T$ is an extreme contraction if and only if either (i) $T=0$, or (ii) $\operatorname{rank}(T)=1$ and at least one diagonal entry of $T$ equals $1$, or (iii) $\operatorname{rank}(T)=2$ and $T$ is an extreme correlation matrix, that is,
\[
T=\begin{bmatrix} 1&a&b\\ a&1&c\\ b&c&1 \end{bmatrix}, \quad\text{where } |a|<1,|b|<1,\mbox{ and }|c|<1.
\]
\end{cor}
\begin{proof}
The sufficient part follows trivially. For the necessary part, let $T$ be an extreme contraction of rank $r$. By Corollary~\ref{cor-02real}, $r\leq 2$. If $r=0$, then $T=0$. If $r=1$, Theorem~\ref{th-grank1} gives (ii). If $r=2$, Corollary~\ref{cor-02real} forces all three diagonal entries to equal $1$, so $T$ is a correlation matrix and hence an extreme correlation matrix. The rest of (iii) follows from Theorem~\ref{th-LT}.
\end{proof}

\subsection{The case $X=\ell_\infty^n$, $Y=\ell_1^n$}

We now turn to extreme contractions of $\mathcal{A}_{\infty\to 1}^{(n)}({\mathbb{R}})$.

\begin{thm} \label{th-norm-inf-1}
Let $T=\big (\!\big ( t_{i j}\big )\! \big ) \in M_3(\mathbb{R})$ be a symmetric non-negative definite matrix. Then
\begin{eqnarray}\label{Formula for positive definite of size 3}
    \|T\|_{\infty \rightarrow 1}=\operatorname{tr}(T)+2\left(\left|t_{12}\right|+\left|t_{13}\right|+\left|t_{23}\right|-2 \min \left\{\left|t_{12}\right|,\left|t_{13}\right|,\left|t_{23}\right|\right\} \cdot \mathbf{1}_{\left\{t_{12} t_{13} t_{23}<0\right\}}\right),
\end{eqnarray}
where $\mathbf{1}_{\left\{t_{12} t_{13} t_{23}<0\right\}}$ equals 1 if $t_{12} t_{13} t_{23}<0$ and 0 otherwise. 
\end{thm}
\begin{proof}
Since $T \geq 0$, the bilinear form $(x, y) \mapsto y^{\top} T x$ is maximized over $\{ \pm 1\}^3 \times\{ \pm 1\}^3$ on the diagonal $y=x$ (see \cite{Rietz}). 
Therefore, 
\begin{eqnarray}\label{Maximum is attained at diagonal}
    \|T\|_{\infty \rightarrow 1}=\max _{x \in\{-1,+1\}^3} x^{\top} T x.
\end{eqnarray}
For $x=\left(x_1, x_2, x_3\right) \in\{ \pm 1\}^3$,
\[
x^{\top} T x=\sum_{i=1}^3 t_{i i} x_i^2+2 \sum_{i<j} t_{i j} x_i x_j=\operatorname{tr}(T)+2\left(t_{12} \varepsilon_{12}+t_{13} \varepsilon_{13}+t_{23} \varepsilon_{23}\right),
\]
where $\varepsilon_{i j}:=x_i x_j \in\{-1,+1\}$. Since $\varepsilon_{12} \varepsilon_{13} \varepsilon_{23}=x_1^2 x_2^2 x_3^2=1$, the triple $\left(\varepsilon_{12}, \varepsilon_{13}, \varepsilon_{23}\right)$ is constrained to the four elements of $\{ \pm 1\}^3$ whose product is $+1$, and conversely. Hence, 
\begin{equation*}
\|T\|_{\infty \rightarrow 1}=\operatorname{tr}(T)+2 M, \quad\text {where} \quad M:=\max _{\substack{\varepsilon_{i j} \in\{ \pm 1\} \\ \varepsilon_{12} \varepsilon_{13} \varepsilon_{23}=1}}\left(\varepsilon_{12} t_{12}+\varepsilon_{13} t_{13}+\varepsilon_{23} t_{23}\right). 
\end{equation*}
We assume that $t_{12}t_{23}t_{13}\neq 0.$ We set $s_{i j}:=\operatorname{sgn}\left(t_{i j}\right) \in \{1,-1\},$ $\eta_{i j}:=\varepsilon_{i j} s_{i j}$, and $\sigma= s_{12}s_{13}s_{23}$. Then we rewrite  
\begin{equation*} 
M=\max _{\substack{\eta_{i j} \in\{ \pm 1\} \\ \eta_{12} \eta_{13} \eta_{23}=\sigma}}\left(\eta_{12}\left|t_{12}\right|+\eta_{13}\left|t_{13}\right|+\eta_{23}\left|t_{23}\right|\right).
\end{equation*}
If $\sigma=+1$,  the admissible triples $\left(\eta_{12}, \eta_{13}, \eta_{23}\right)$ are $(+,+,+),(+,-,-),(-,+,-),(-,-,+)$, giving values
\[
|t_{12}|+|t_{13}|+|t_{23}|, \quad |t_{12}| -|t_{13}| - |t_{23}| , \quad -|t_{12}| + |t_{13}| - |t_{23}|,\quad - |t_{12}| - |t_{13}| + |t_{23}|.
\]
If $\sigma=-1$, then the admissible triples $\left(\eta_{12}, \eta_{13}, \eta_{23}\right)$ are giving values 
\[
-|t_{12}| - |t_{13}| - |t_{23}|, \quad -|t_{12}| + |t_{13}| + |t_{23}|, \quad |t_{12}| - |t_{13}| + |t_{23}|, \quad |t_{12}| + |t_{13}| - |t_{23}|.
\]
The first is clearly not maximal. The maximum of the remaining three equals
\begin{multline*}
\max \{|t_{13}| + |t_{23}| - |t_{12}|, |t_{12}| + |t_{23}| - |t_{13}|, |t_{12}|  + |t_{13}| - |t_{23}| \}\\ = (|t_{12}| + |t_{13}| + |t_{23}| )-2 \min \{|t_{12}| , |t_{13}| , |t_{23}| \}.
\end{multline*}
Combining the two cases, we have 
\begin{equation}\label{amar1st}
M=|t_{12}|+|t_{13}|+|t_{23}|-2 \min \{|t_{12}|,|t_{13}|,|t_{23}|\} \cdot \mathbf{1}_{\{t_{12} t_{13} t_{23}<0\}}.\end{equation}
If some $t_{i j}=0$, then $t_{12} t_{13} t_{23}=0$, the indicator is $0$,
and the expression for $M$ reduces to $\left|t_{12}\right|+\left|t_{13}\right|+\left|t_{23}\right|$, i.e., the sum of the absolute values of the two nonzero off-diagonal entries. However, observe that if one of the off-diagonal, say $t_{23}=0$, the constraint $\varepsilon_{12} \varepsilon_{13} \varepsilon_{23}=1$ can be satisfied while independently choosing $\varepsilon_{12}= s_{12}$ and $\varepsilon_{13}=s_{13}$ by setting $\varepsilon_{23}=\varepsilon_{12} \varepsilon_{13}$. Thus, each surviving $\left|t_{i j}\right|$ is attained with a $+$ sign. This shows that the formula in \eqref{amar1st} is valid in this case also. Therefore, by substituting the value of $M$ into $\|T\|_{\infty \rightarrow 1}=\operatorname{tr}(T)+2 M$ yields that
\[
\|T\|_{\infty \rightarrow 1}=\operatorname{tr}(T)+2\left(\left|t_{12}\right|+\left|t_{13}\right|+\left|t_{23}\right|-2 \min \left\{\left|t_{12}\right|,\left|t_{13}\right|,\left|t_{23}\right|\right\} \cdot \mathbf{1}_{\left\{t_{12} t_{13} t_{23}<0\right\}}\right),
\]
as claimed. This completes the proof of the theorem.
\end{proof}
Even though Theorem \ref{th-norm-inf-1} is written 
for symmetric non-negative definite matrices only, the proof shows that the explicit formula for the norm of a symmetric matrix is given by \eqref{Formula for positive definite of size 3} remains valid as long as we have the equality \eqref{Maximum is attained at diagonal}. For example, consider $T=(\!(t_{ij})\!)_{3\times 3}$ with $t_{ij}=1$ for all $(i,j)$ except $(i,j)=(2,3)$ and $(i,j)=(3,2),$ and $t_{23}=t_{32}=-1.$ Note that $\|T\|_{\infty \to 1}=5,$ which agrees with the right hand side of \eqref{Formula for positive definite of size 3}.

For the proof of the next theorem, we need the formula for $\|T\|_{\infty \to 1}$, where $T\in \mathcal{M}_2(\mathbb{C})$ is self-adjoint. Such a formula, given in the lemma below, is obtained by a straightforward but lengthy computation which is given in the Appendix.
\begin{lem}\label{norm:lem 2.17} Let $a, b \in \mathbb{R}$ and $r \geq 0$. Define the matrix $A = \begin{pmatrix} a & re^{i\theta} \\ re^{-i\theta} & b \end{pmatrix}$ for $0 \leq \theta < 2\pi$. The matrix norm $\|A\|_{\infty\to 1}$ is given by:
\[ \|A\|_{\infty\to 1} = \sup_{x \in [0, 2\pi)} f(x) \]
where the objective function $f(x)$ is defined as:
\[ f(x) = \sqrt{a^2 + r^2 + 2ar \cos x} + \sqrt{b^2 + r^2 + 2br \cos x} \]
The value of this supremum is characterized as follows:
\begin{enumerate}
    \item If $ab \geq 0$, the maximum occurs at $x=0$, yielding $\|A\|_{\infty\to 1} = |a| + |b| + 2r$.
    \item If $ab < 0$, let $s = -b$ and assume $a, s > 0$. The supremum is $\max \{ f(0), f(\pi), f(x^*) \}$, where $x^*$ is a critical point satisfying $\cos x^* = \frac{r(a-s)}{2as}$ and $f(x^*) = (a+s)\sqrt{1 + \frac{r^2}{as}}$, provided $\left| \frac{r(a-s)}{2as} \right| \leq 1$. 
\end{enumerate}
In the case $b = -a$, the norm is $2\sqrt{a^2 + r^2}$.
\end{lem}
\begin{thm}\label{th-2dim}
Let $T\in \mathcal{A}_{\infty\to 1}^{(2)}(\mathbb F)$ be non-zero. Then the following are equivalent.

\noindent (i) $T$ is an extreme contraction.

\noindent (ii) $\operatorname{rank}(T)=1$ and $a+b+2r=1$, where
$T=\begin{bmatrix} a & re^{i\theta} \\ re^{-i\theta}& b \end{bmatrix}$, with $a,b,r\geq 0$ and $0\leq \theta< 2\pi$ if $\mathbb{F}=\mathbb{C}$, and $\theta\in \{0,\pi\}$ if $\mathbb{F}=\mathbb{R}$.
\end{thm}
\begin{proof}
$(ii)\Rightarrow (i)$ follows from Theorem~\ref{th-grank1} and the fact that $\|T\|_{\infty\to 1}=a+b+2r$.

$(i)\Rightarrow (ii)$. Since $T$ is an extreme contraction, Proposition~\ref{prop-g1} gives $\|T\|_{\infty\to 1}=1$, that is, $a+b+2r=1$. We show $\operatorname{rank}(T)=1$. Suppose $\operatorname{rank}(T)=2$. Then $a>0$, $b>0$, and $ab-r^2>0$. Choose $\epsilon>0$ such that $a\pm \epsilon>0$, $b\pm \epsilon>0$, and $ab-r^2\pm \epsilon(b-a\mp \epsilon)>0$. Consider
\[
T_1=\begin{bmatrix} a +\epsilon& re^{i\theta} \\ re^{-i\theta}& b-\epsilon \end{bmatrix}, \quad T_2=\begin{bmatrix} a -\epsilon& re^{i\theta} \\ re^{-i\theta}& b+\epsilon \end{bmatrix}.
\]
Then $\|T_1\|_{\infty\to 1}=\|T_2\|_{\infty\to 1}=a+b+2r=1$ by Lemma \ref{norm:lem 2.17}, both $T_1\geq 0$ and $T_2\geq 0$, and $T=\frac{1}{2}(T_1+T_2)$, contradicting extremality. Therefore $\operatorname{rank}(T)=1$.
\end{proof}

\begin{lem}\label{lem-directsum}
Let $T=A\oplus 0$ with $A\in \mathcal{A}_{\infty\to 1}^{(n)}(\mathbb F)$. Then $T$ is an extreme contraction of  $\mathcal{A}_{\infty\to 1}^{(n+1)}(\mathbb F)$ if and only if $A$ is an extreme contraction of $\mathcal{A}_{\infty\to 1}^{(n)}(\mathbb F)$.
\end{lem}
\begin{proof}
Observe that $\|T\|_{\infty\to 1}=\|A\|_{\infty\to 1}$. If $T$ is extreme and $A=\frac{1}{2}A_1+\frac{1}{2}A_2$ for some $A_1,A_2\in \mathcal{A}_{\infty\to 1}^{(n)}$, then $T=\frac{1}{2}(A_1\oplus 0)+\frac{1}{2}(A_2\oplus 0)$, forcing $A_1\oplus 0=A_2\oplus 0$, hence $A=A_1=A_2$.

Conversely, if $A$ is extreme and $T=\frac{1}{2}(T_1+T_2)$ for some $T_1,T_2\in \mathcal{A}_{\infty\to 1}^{(n+1)}$, then $(T)_{(n+1)(n+1)}=0$ implies $(T_i)_{(n+1)(n+1)}=0$ for $i=1,2$, so $T_i=A_i\oplus 0$ and $A=\frac{1}{2}(A_1+A_2)$. Since $A$ is extreme, $A_1=A_2$, hence $T_1=T_2$.
\end{proof}

\begin{thm}\label{thm-rank2-inftyto1}
 Suppose $T=(t_{ij})_{i,j=1}^3\in \mathcal{A}_{\infty\to 1}^{(3)}(\mathbb{R})$ with $\operatorname{rank}(T)=2$. If  $T$ is an extreme contraction of $\mathcal{A}_{\infty\to 1}^{(3)}(\mathbb{R})$, then either (i) $t_{12}=t_{13}=t_{23}<0$, or (ii) $|t_{12}|=|t_{13}|=|t_{23}|$ and exactly two of $t_{12},t_{13},t_{23}$ are positive.
\end{thm}
\begin{proof}
Suppose  $T$ is an extreme contraction of $\mathcal{A}_{\infty\to 1}^{(3)}(\mathbb{R}).$ Thus $\|T\|_{\infty\to 1} = 1$. Without loss of generality, we may assume that $T$ is a non-diagonal matrix; otherwise using Lemma \ref{lem-directsum} and Theorem \ref{th-2dim}, we can show that $T$ can not be an extreme contraction. Since $T$ is non-negative definite, the norm is attained at some $z = (z_1, z_2, z_3) \in \{-1, 1\}^3$ such that $\langle Tz, z \rangle = 1$. Since $z$ is a point of maxima, therefore, using $\langle Tz, z \rangle\geq \langle T(-z_1,z_2,z_3), (-z_1,z_2,z_3) \rangle,$ we must have $t_{12}z_1z_2+t_{13}z_1z_3 \geq  0.$ 
Similarly, we must also have $t_{12}z_1z_2+t_{23}z_2z_3 \geq 0$ and 
$t_{13}z_1z_3+t_{23}z_2z_3 \geq 0.$
We first show that at least two of the following equations must hold:
\begin{align}
t_{12}z_2+t_{13}z_3 &= 0, \label{eq-001}\\
t_{12}z_1+t_{23}z_3 &= 0, \label{eq-002}\\
t_{13}z_1+t_{23}z_2 &= 0. \label{eq-003}
\end{align}
To verify this claim, assume, for instance, to the contrary that  \eqref{eq-001} and \eqref{eq-002} do not hold, that is, 
\begin{align}
    t_{12}z_1z_2 + t_{13}z_1z_3 &> 0 \label{eq-004}\\
    t_{12}z_1z_2+t_{23}z_2z_3 &> 0. \label{eq-005}
\end{align}
Since $T$ is non-negative definite and $\operatorname{rank}(T)$ is $2$, we may fix a factorization
\[
T=A\Gamma A^*,
\]
where $A\in M_{3,2}(\mathbb{R})$ has rank $2$ and 
$\Gamma\in M_2(\mathbb{R})$ is positive definite. 
We shall construct a non-zero self-adjoint matrix
\[C=\begin{pmatrix}
	c_{11}&c_{12}\\
	c_{12}&c_{22}\\
	\end{pmatrix}
\]
such that, defining
\[
E=ACA^*=(e_{ij}),
\]
we have
\[
T\pm E=A(\Gamma\pm C)A^*\geq 0
\]
and
\[
\|T\pm E\|_{\infty\to1}=\|T\|_{\infty\to1}.
\]
This will show that
\[
T=\frac12(T+E)+\frac12(T-E)
\]
is a non-trivial convex decomposition, contradicting the extremality of $T$.

Observe that the following equations have a non-zero solution.
\begin{equation}\label{eq-007}
\operatorname{tr}(E)+2e_{12}z_1z_2=0, \quad e_{13}z_1z_3+e_{23}z_2z_3=0.	
	\end{equation}
Choose a non-zero solution of (\ref{eq-007}), where $c_{11},c_{12},c_{22}$ are sufficiently small so that  $\Gamma \pm C\geq 0$ and 
\begin{equation}\label{eq-008}
 (t_{12}\pm e_{12})z_1z_2+(t_{13}\pm e_{13})z_1z_3 >  0, \quad  (t_{12}\pm e_{12})z_1z_2+(t_{23}\pm e_{23})z_2z_3 >  0.
 \end{equation}
Note that the last two inequalities hold for sufficiently small $c_{11},c_{12},c_{22}$ because of (\ref{eq-004}) and (\ref{eq-005}).
Now,
\begin{eqnarray*}
	\langle (T+E)z,z\rangle &=& \operatorname{tr}(T+E)+2((t_{12}+e_{12})z_1z_2+(t_{13}+e_{13})z_1z_3+(t_{23}+e_{23})z_2z_3) \\
	&=&\operatorname{tr}(T)+2(t_{12}z_1z_2+t_{13}z_1z_3+t_{23}z_2z_3), ~(\text{by using } (\ref{eq-007}))\\
	&=& \|T\|_{\infty\to 1}.
	\end{eqnarray*}
In what follows, we adopt the notation: $\sigma^{(i)} \odot z$, $i=1,2,3$, where 
$\sigma^{(1)}=(-1,1,1)^\top$, $\sigma^{(2)}=(1,-1,1)^\top$, $\sigma^{(3)}=(1,1,-1)^\top$, $z=(z_1, z_2, z_3)^T$ and $\odot$ is the Hadamard product. 
    From (\ref{eq-008}), it follows that
\begin{align*}
\big\langle(T+E) 
	\sigma^{(1)}\odot z, &\,\, \sigma^{(1)}\odot z\big\rangle  \\ 
    &= \operatorname{tr}(T+E)+2(-(t_{12}+e_{12})z_1z_2-(t_{13}+e_{13})z_1z_3+(t_{23}+e_{23})z_2z_3) \\
	&< \operatorname{tr}(T+E)+2((t_{12}+e_{12})z_1z_2+(t_{13}+e_{13})z_1z_3+(t_{23}+e_{23})z_2z_3)\\
	&= \|T\|_{\infty\to 1}.\phantom{ tr(T+E)+2(-(t_{12}+e_{12})z_1z_2-(t_{13}+e_{13})z_1z_3+(t_{23}+e_{23})z_2z_3)}.
\end{align*}
Similarly, using (\ref{eq-008}), we have 
\begin{align*}
	\big \langle (T+E) \sigma^{(2)}\odot z, &\,\, \sigma^{(2)} \odot z \big \rangle\\  & = \operatorname{tr}(T+E)+2(-(t_{12}+e_{12})z_1z_2+(t_{13}+e_{13})z_1z_3-(t_{23}+e_{23})z_2z_3) \\
	&< \operatorname{tr}(T+E)+2((t_{12}+e_{12})z_1z_2+(t_{13}+e_{13})z_1z_3+(t_{23}+e_{23})z_2z_3)\\
	&= \|T\|_{\infty\to 1}.
\end{align*}
Finally, 
\begin{align*}
	\big \langle (T+E) \sigma^{(3)}\odot z,&\,\, \sigma^{(3)} \odot z \big \rangle\\ &= \operatorname{tr}(T+E)+2((t_{12}+e_{12})z_1z_2-(t_{13}+e_{13})z_1z_3-(t_{23}+e_{23})z_2z_3) \\
	&= \operatorname{tr}(T) + \operatorname{tr}(E) + 2e_{12}z_1z_2 - 2(e_{13}z_1z_3 + e_{23}z_2z_3) \\
    &\phantom{\qquad\quad}+2(t_{12}z_1z_2 - t_{13}z_1z_3 - t_{23}z_2z_3)\\
	&= \operatorname{tr}(T)+2(t_{12}z_1z_2 - t_{13}z_1z_3 - t_{23}z_2z_3) \quad (\text{using } (\ref{eq-007})) \\
	&= \langle T (\sigma^{(3)}\odot z), \sigma^{(3)}\odot z \rangle \\
	&\leq \|T\|_{\infty\to 1}.
\end{align*}
Since $\langle (T+E)z,z\rangle = \|T\|_{\infty\to 1}$ and the values at all other extreme points are less than or equal to $\|T\|_{\infty\to 1}$, we conclude that 
$\|T+E\|_{\infty\to 1}=\|T\|_{\infty\to 1}.$
Proceeding similarly, we can show that $\|T-E\|_{\infty\to 1}=\|T\|_{\infty\to 1}.$ This shows that $T$ is not an extreme contraction. Hence we conclude that at least two of (\ref{eq-001}), (\ref{eq-002}) and (\ref{eq-003}) must hold. Thus, $|t_{12}| = |t_{13}| = |t_{23}| = k$ for some $k > 0$.

If $t_{12}=t_{13}=t_{23}>0$, the norm is attained at $z=\pm(1,1,1)$ and none of \eqref{eq-001}--\eqref{eq-003} is satisfied, therefore $T$ can not be an extreme contraction. Similarly, cases with exactly two negatives among $t_{12}, t_{13}, t_{23}$ (when all have the same absolute value) fail.
\end{proof}

\begin{thm}\label{thm-normattain}
Suppose $T\in \mathcal{A}_{\infty\to 1}^{(n)}(\mathbb{R})$ is an extreme contraction of rank $r$. Then $T$ attains its norm at no fewer than $r^2+r$ extreme points of $(\ell_\infty^n)_1$.
\end{thm}
\begin{proof} The extreme points of the unit ball $(\ell_\infty^n)_1$ are the vectors in $\{-1, 1\}^n$. By symmetry, they occur in antipodal pairs $\pm w$. Let $S$ be a set of $2^{n-1}$ representative extreme points formed by fixing the final coordinate to $1$:
\[S = \{w_j = (w_{j1}, \dots, w_{j(n-1)}, 1) : w_{jm} \in \{1, -1\}\}.\]
Suppose for the sake of contradiction that $T$ attains its norm at strictly fewer than $r^2+r$ extreme points. Since these points come in $\pm$ pairs, $T$ attains its norm at $2k$ points where $2k < r^2+r$. We can order the vectors in $S$ such that $T$ attains its norm at the first $k$ vectors. Therefore, for the norm-attaining points,

\[\langle Tw_i, w_i \rangle = \|T\|_{\infty\to 1}, \quad \forall~ 1 \leq i \leq k,\]
and for the remaining non-norm-attaining points in $S$,

\[\langle Tw_j, w_j \rangle < \|T\|_{\infty\to 1}, \quad \forall~ k+1 \leq j \leq 2^{n-1}.\]
Because $2k < r^2+r$, it follows that $k < \frac{r^2+r}{2}$. Since $T \geq 0$ and has rank $r$, we can factor it as $T = AA^T$, where $A$ is an $n \times r$ real matrix with full column rank. We seek to construct a small symmetric perturbation of the form $E = ACA^T$, where $C$ is an $r \times r$ real symmetric matrix. The real vector space of $r \times r$ symmetric matrices has dimension $\frac{r^2+r}{2}$. We require our perturbation to satisfy

\[\langle Ew_i, w_i \rangle = 0, \quad \forall~ 1 \leq i \leq k.\]
Substituting $E = ACA^T$, this requirement becomes$$\langle ACA^T w_i, w_i \rangle = (A^T w_i)^T C (A^T w_i) = 0, \quad \forall~ 1 \leq i \leq k.$$This imposes $k$ homogeneous linear equations on the entries of $C$. Because the number of equations $k$ is strictly less than the dimension of the space, there exists a non-trivial solution. Thus, we can choose a non-zero symmetric matrix $C$, which gives a non-zero symmetric perturbation matrix $E$. Because $T = AA^T$ and $E = ACA^T$, we can scale $C$ (and consequently $E$) to be arbitrarily small while still satisfying our homogeneous equations. We scale $E$ to be small enough to guarantee two conditions: first, that $T \pm E = A(I_{r \times r} \pm C)A^T \geq 0$, and second, that the strict inequalities separating the norm from the non-norm-attaining points are preserved. Specifically, for all $k+1 \leq j \leq 2^{n-1}$, we ensure

\[\langle (T\pm E) w_1, w_1 \rangle > \langle (T\pm E) w_j, w_j \rangle.\]
We now evaluate the norm of the perturbed operators $T \pm E$. For the $k$ norm-attaining points, the perturbation vanishes by design:
\[\langle (T\pm E) w_i, w_i \rangle = \langle Tw_i, w_i \rangle \pm \langle Ew_i, w_i \rangle = \|T\|_{\infty\to 1} \pm 0 = \|T\|_{\infty\to 1}.\]
For the remaining points, the strict inequalities established above ensure that their values remain strictly bounded below $\|T\|_{\infty\to 1}$. Consequently, the norm of the perturbed operators is completely determined by the first $k$ points:

\[\|T \pm E\|_{\infty\to 1} = \sup_{1 \leq j \leq 2^{n-1}} \langle (T\pm E) w_j, w_j \rangle = \langle (T\pm E) w_1, w_1 \rangle = \|T\|_{\infty\to 1}.\]
We have shown that $\|T + E\|_{\infty\to 1} = \|T - E\|_{\infty\to 1} = \|T\|_{\infty\to 1}$. Because $E$ is non-zero, we can write $T$ as the non-trivial convex combination

\[T = \frac{1}{2} \bigl((T + E) + (T - E)\bigr).\]
This implies that $T$ is not an extreme contraction, which contradicts our initial assumption. Therefore, $T$ must attain its norm at at least $r^2+r$ extreme points.
\end{proof}

\begin{cor}\label{infty to 1 - low dimension - rank bound}
 If $T\in \mathcal{A}_{\infty\to 1}^{(n)}(\mathbb{R})$ is an extreme contraction and $n=4$, then $\operatorname{rank}(T)\leq 3$; if $n=3$, then $\operatorname{rank}(T)\leq 2$.
\end{cor}

\subsection{The positive Grothendieck constant $K_G^{+,\mathbb{F}}(3)$}\label{sec:KG}
Let $T=(t_{ij})\in M_n(\mathbb{R})$ be symmetric and non-negative definite. The $\ell_\infty\to\ell_1$ norm is given by
\[
\|T\|_{\infty\to 1} = \max_{\varepsilon\in\{\pm1\}^n} \varepsilon^T T\varepsilon = \operatorname{tr}(T)+2\max_{\varepsilon\in\{\pm1\}^n}\sum_{i<j} t_{ij}\varepsilon_i\varepsilon_j,
\]
where the first equality follows from $T\geq 0$ as in the proof of Theorem \ref{th-norm-inf-1}. This norm is equivalent to the cut norm, defined by
\[
\|T\|_{\mathrm{cut}} = \max_{S,S'\subseteq [n]}\Big|\sum_{i\in S,\,j\in S'} t_{ij}\Big|,
\]
where $S,S'$ are any two subsets of $[n]$. They satisfy the relation $\|T\|_{\mathrm{cut}} \leq \|T\|_{\infty\to 1} \leq 4\,\|T\|_{\mathrm{cut}}$, see \cite[p.~788]{AN}.

Computing these norms is NP-hard in general. For instance, if $L=D-A$ is the Laplacian of a graph $G=(V,E)$, then for any $\varepsilon\in\{\pm1\}^n$,
\[
\varepsilon^T L\varepsilon = \sum_{\{i,j\}\in E}(\varepsilon_i-\varepsilon_j)^2 = 4\,|E(S,S^c)|,
\]
where $S=\{i:\varepsilon_i=1\}$. Maximizing $|E(S,S^c)|$ is the \textsc{Max-Cut} problem; hence computing $\|L\|_{\infty\to 1}$ (and by extension $\|L\|_{\mathrm{cut}}$) is NP-hard.

To approximate this value, the natural SDP relaxation is introduced. For a symmetric matrix $\big(\!\big ( t_{i,j}\big )\!\big)$, define 
\[SDP(T) = \sup\big \{ \sum_{i,j} t_{ij}\langle u_i,u_j\rangle: u_1, \ldots, u_n, \|u_i\|=1\big\},\] 
which replaces the signs $\varepsilon_i\in\{\pm1\}$ with unit vectors $u_i$ in a Hilbert space $H$. The positive Grothendieck constant $K_G^+$ controls the integrality gap:
\[
\text{SDP}(T) \leq K_G^+\cdot \|T\|_{\infty\to 1}.
\]

For $n=3$, only four sign vectors are admissible (after fixing $\varepsilon_1=1$), and their structure is governed by the single parity bit $\sigma=\operatorname{sgn}(t_{12}t_{13}t_{23})$ --- simple enough to optimize in closed form, as Theorem \ref{th-norm-inf-1} shows. As far as we know, no comparable reduction is available for $n\geq 4$.

Let $A = (a_{ij})_{i,j=1}^n\in M_n(\mathbb F)$. Consider
\begin{equation} \label{eqn:14}
\Gamma(A):= \sup\Big\{\Big|\sum_{i,j=1}^n a_{ij}\langle v_i,w_j\rangle\Big|:\ \|v_i\|_2=1,\ \|w_j\|_2=1\ \text{for all }i,j\Big\},
\end{equation}
where $v_i,w_j$ range over vectors in arbitrary Hilbert space $\mathcal{H}$ over $\mathbb F.$
Define the numerical constant
\[
K_G^{\mathbb F}(n) \stackrel{\text { def }}{=} \sup \left\{\Gamma(A): A\in M_n(\mathbb F),\|A\|_{\infty \to 1} \leq 1\right\}.
\]
The constant $K_G^{\mathbb F}(n)$ clearly depends on the ground field. The fact that $K_G^{\mathbb F}(n)$ remains finite as $n \rightarrow \infty$ was established by Grothendieck. The limit of this sequence is denoted by $K_G^{\mathbb F}$, and is called the real or complex Grothendieck constant depending on the scalar field $\mathbb{F}$ being real or complex. Its exact value is not known. The limit taken over non-negative definite matrices is finite as well and is denoted by $K_G^{+,\mathbb F}$. 

{An immediate consequence of Theorem \ref{th-2dim} is that $K_G^{+, \mathbb{R}}(2) = 1$}. The theorem below gives the exact value of $K_G^{+, \mathbb{R}}(3)$.
\begin{thm}\label{thm-KG}
$K_G^{+,\mathbb{R}}(3)=\frac{9}{8}$.
\end{thm}
\begin{proof}
Let $E_{1\to \infty}$ and $E_{\infty\to 1}$ denote the sets of extreme contractions of $\mathcal{A}_{1\to\infty}^{(3)}(\mathbb{R})$ and $\mathcal{A}_{\infty\to 1}^{(3)}(\mathbb{R})$, respectively. 
Then
\begin{align*}
K_G^{+,\mathbb{R}}(3) &:= \sup\Big\{\Big|\sum_{i,j=1}^3 b_{ij}\langle x_i,x_j\rangle\Big|: B=(b_{ij})\geq 0,\, \|B\|_{\infty\to 1}\leq 1,\, \|x_i\|_2\leq 1\Big\}\\
&= \sup\{|\langle A,B\rangle|: A\in \mathcal{A}_{1\to \infty}^{(3)},\, B\in \mathcal{A}_{\infty\to 1}^{(3)}\}\\
&= \sup\{|\langle A,B\rangle|: A\in E_{1\to \infty},\, B\in E_{\infty\to 1}\}.
\end{align*}
Clearly $K_G^{+,\mathbb{R}}(3)\geq 1$. If $B=(b_{ij})$ satisfies $\|B\|_{\infty\to 1}=\sum_{i,j=1}^3|b_{ij}|$, then $|\sum_{i,j=1}^3 b_{ij}\langle x_i,x_j\rangle|\leq 1$. Thus these matrices don't count in the computation of  $K_G^{+,\mathbb{R}}(3)$.
By the norm computation in Theorem~\ref{th-norm-inf-1}, either all of $\{b_{12},b_{13},b_{23}\}$ are negative or exactly one is negative for any  extremal $B$.

Note that, from Corollary \ref{infty to 1 - low dimension - rank bound}, $\operatorname{rank}(B)\leq 2$. We first show $\operatorname{rank}(B)=2$. If $\operatorname{rank}(B)=1$, then $B=\gamma\begin{bmatrix} 1&\delta&\mu\\ \delta&\delta^2&\delta\mu\\ \mu&\delta\mu&\mu^2 \end{bmatrix}$ with $\gamma>0$. The sign conditions force (without loss of generality) $\mu=0$, $\delta<0$, and $\|B\|_{\infty\to 1}=\gamma(1-2\delta+\delta^2)$. Then $|\langle A,B\rangle|=\gamma|a_{11}+2a_{12}\delta+a_{22}\delta^2|\leq \gamma(1-2\delta+\delta^2)=1$, so the supremum is not attained at rank-$1$ matrices.

Therefore $\operatorname{rank}(B)=2$ and $B\in E_{\infty\to 1}$. By Theorem~\ref{thm-rank2-inftyto1}, either (I) $b_{12}=b_{13}=b_{23}<0$, or (II) $|b_{12}|=|b_{13}|=|b_{23}|$ with exactly two positive and one negative. In both cases $\|B\|_{\infty\to 1}=\operatorname{tr}(B)+2|b_{12}|$.

Consider the case $B=\begin{bmatrix} b_{11}&-b_{12}&b_{12}\\ -b_{12}&b_{22}&b_{12}\\ b_{12}&b_{12}&b_{33} \end{bmatrix}$ with $b_{12}>0$. Since $B\geq 0$ and $\operatorname{rank}(B)=2$, writing $B$ as the Gram matrix of $\{x_1,x_2,\alpha x_1+\beta x_2\}$ with linearly independent $x_1,x_2$ and $\alpha,\beta>0$, the constraint $\|B\|_{\infty\to 1}=1$ gives
\[
b_{12}=\frac{\alpha\beta}{(\alpha+\beta)(1+\alpha)(1+\beta)}.
\]
One checks that $\max\big\{\frac{\alpha\beta}{(\alpha+\beta)(1+\alpha)(1+\beta)}:\alpha>0,\beta>0\big\}=\frac{1}{8}$.

For the extremal $A$, rank-$1$ matrices yield $|\langle A,B\rangle|\leq 1$. So $\operatorname{rank}(A)=2$. By Corollary~\ref{cor-realn=3}, $A$ is an extreme correlation matrix:
\[
A=\begin{bmatrix} 1& a&\gamma+\delta a\\ a&1&\gamma a+\delta\\ \gamma +\delta a & \gamma a+\delta & 1 \end{bmatrix},
\]
where $\gamma^2+\delta^2+2a\gamma\delta=1$, $|a|<1$, $|\gamma+\delta a|<1$, $|\gamma a+\delta|<1$.

Computing $\langle A,B\rangle = \operatorname{tr}(B)+2b_{12}(-a+\gamma+\delta+a\gamma+a\delta)$, one obtains
\[
\langle A,B\rangle \leq \operatorname{tr}(B)+3b_{12} = 1+b_{12} \leq 1+\frac{1}{8} = \frac{9}{8},
\]
and $\langle A,B\rangle \geq \operatorname{tr}(B)-6b_{12} = 1-8b_{12} \geq 0$. The other cases of sign patterns in $B$ yield the same bound.

The value $\frac{9}{8}$ is attained by
\[
A=\begin{bmatrix} 1& -\frac{1}{2}&\frac{1}{2}\\ -\frac{1}{2}&1&\frac{1}{2}\\ \frac{1}{2}& \frac{1}{2}& 1 \end{bmatrix}, \quad B=\begin{bmatrix} \frac{1}{4}& -\frac{1}{8}&\frac{1}{8}\\ -\frac{1}{8}&\frac{1}{4}&\frac{1}{8}\\ \frac{1}{8}& \frac{1}{8}& \frac{1}{4} \end{bmatrix},
\]
giving $\langle A, B \rangle =\frac{9}{8}$.
\end{proof}
The computation below giving a lower bound for the Grothendieck constant in the real case is not sharp. However, obtaining the lower bound using Khintchine's inequality is immediate. We therefore include it here. First, define
\[
\widetilde{\Gamma}(A) 
:= \sup\Big\{\Big|\sum_{i,j=1}^n a_{ij}\langle v_i,e_j\rangle\Big|:\ \|v_i\|_2=1,\ 1\le i\le n\Big\},
\]
where $(e_j)$ is an orthonormal system in $\mathcal{H}$. Clearly, $\widetilde{\Gamma}(A)\leq{\Gamma}(A)$, see Equation \eqref{eqn:14}.

Note that we also have $\widetilde{\Gamma}(A)\leq K_G^{\mathbb{R}}\|A\|_{\infty\to 1}.$  Let $\widetilde{K_G^{\mathbb{R}}}$ denote the best possible constant in this inequality.
\begin{prop} 
$\displaystyle \widetilde{K_G^{\mathbb{R}}} = \sqrt{2}.$
\end{prop}

\begin{proof}
By the duality $(\ell_1^n)^*\cong\ell_\infty^n$, we have 
\[
\|A\|_{\infty\to 1}
=\sup\Big\{\sum_{i=1}^n\Big|\sum_{j=1}^n a_{ij}y_j\Big|:\ y_j\in\{+1,-1\}\Big\}.
\]
On the other hand,
\begin{align*}
\widetilde{\Gamma}(A)
&=\sup\Big\{\Big|\sum_{i,j=1}^n a_{ij}\langle v_i,e_j\rangle\Big|:\ \|v_i\|_2=1,\ 1\le i\le n\Big\}\\
&=\sup\Big\{\Big|\sum_{i=1}^n \Big\langle v_i,\sum_{j=1}^n a_{ij}e_j\Big\rangle\Big|:\ \|v_i\|_2=1,\ 1\le i\le n\Big\}.
\end{align*}
For each $i$, the inner supremum is attained when $v_i$ is in the direction of $\sum_{j=1}^n a_{ij}e_j$. Hence,
\[
\widetilde{\Gamma}(A)=\sum_{i=1}^n\Big(\sum_{j=1}^n|a_{ij}|^2\Big)^{1/2}.
\]

Now, let $\varepsilon_1,\dots,\varepsilon_n$ be independent Rademacher random variables taking values $\pm1$ with equal probability. By the sharp Khintchine inequality, for each $i$,
\[
\mathbb{E}\Big|\sum_{j=1}^n a_{ij}\varepsilon_j\Big|
\ge \frac{1}{\sqrt{2}}\Big(\sum_{j=1}^n |a_{ij}|^2\Big)^{1/2},
\]
see \cite{Haagerup1981}. Summing over $i$ gives
\[
\mathbb{E}\sum_{i=1}^n\Big|\sum_{j=1}^n a_{ij}\varepsilon_j\Big|
\ge \frac{1}{\sqrt{2}}\sum_{i=1}^n\Big(\sum_{j=1}^n |a_{ij}|^2\Big)^{1/2}.
\]
Hence, there exists a particular choice of signs $(\varepsilon_j)\in\{\pm1\}^n$ such that
\[
\sum_{i=1}^n\Big|\sum_{j=1}^n a_{ij}\varepsilon_j\Big|
\ge \frac{1}{\sqrt{2}}\sum_{i=1}^n\Big(\sum_{j=1}^n |a_{ij}|^2\Big)^{1/2}.
\]
Taking the supremum over all $y_j\in\{\pm1\}$, we obtain
\[
\|A\|_{\infty\to1}\ge \frac{1}{\sqrt{2}}\,\widetilde{\Gamma}(A),
\]
or equivalently,
\[
\widetilde{\Gamma}(A)\le \sqrt{2}\,\|A\|_{\infty\to1}.
\]
Thus, $\widetilde{K_G^{\mathrm{R}}} \leqslant \sqrt{2}$.

 To show that $\sqrt{2}$ is also a lower bound, consider the $2 \times 2$ matrix $A=\frac{1}{2} H_2$, where $H_2= \left ( \begin{smallmatrix}1 & \,\,1 \\ 1 & -1\end{smallmatrix}\right)$ is the standard Hadamard matrix. For any vector $e \in \{ \pm 1\}^2$, the vector $A e$ is always a permutation of $( \pm 1,0)$, hence $\|A e\|_1=1$, which establishes $\|A\|_{\infty \rightarrow 1}=1$. Since each row $w_i$ has entries $\pm 1 / 2$, its Euclidean norm is $\left\|w_i\right\|_2=1 / \sqrt{2}$. The sum of the row norms is thus $2 \cdot(1 / \sqrt{2})=\sqrt{2}$, showing the bound is attained.
\end{proof}
Since $\widetilde{K_G^{\mathbb{R}}} \leqslant  K_G^{\mathbb{R}}$, we have the following corollary.
\begin{cor}
  $K_G^{\mathbb{R}} \geqslant \sqrt{2}$. 
\end{cor}

\section{Property Q}\label{sec:propertyQ}
Let $X$ be a finite-dimensional normed linear space, and let $A\in X\otimes X$. Since $\dim X = n < \infty$, choosing a basis $e_1,\ldots ,e_n$ of $X$, any $A\in X\otimes X$ can be written as $A=\sum_{i,j=1}^n a_{ij}\, e_i \otimes e_j$ for some scalars $a_{ij}$. We say that  $A$ is non-negative definite, denoted by $A\ge 0$, if the matrix $\big (\!\big (a_{ij}\big )\!\big)_{i,j=1}^n$ is non-negative definite. Equivalently, $A$ lies in the convex hull of the tensors $x \otimes \bar{x}$, $x \in X$, where if $x=\sum_{i=1}^n x_i e_i$, then $\bar{x}= \sum_{i=1}^n \bar{x}_i\, e_i$.

If $A \in X\otimes X$, it defines a linear transformation from $X^*$ to $X$ by $A (\lambda) = \sum_{i,j=1}^n a_{ij} \lambda(e_i) e_j$, $\lambda \in X^*$. The \textit{injective norm} $\|A\|_\epsilon$ of $A$ is defined as
\begin{equation}\label{inj}
\|A\|_\epsilon := \sup\{ \|A (\lambda)\|_X : \lambda\in X^*,\, \|\lambda\|_{X^*} = 1\}. 
\end{equation}
We let $X\otimes_\epsilon X$ denote $X\otimes X$ equipped with the injective norm. 
The \textit{projective norm} (the dual of the injective norm in $X^*\otimes X^*$) is given by
\begin{equation}\label{proj}
\|A\|_\pi:= \inf \Big\{\sum_{k=1}^N \|v_k\| \|w_k\| : A = \sum_{k=1}^N v_k\otimes w_k,\, v_k, w_k \in X,\, N\in\mathbb{N}\Big\},
\end{equation}
where the infimum is taken over all representations of $A$ as a finite sum of elementary tensors. Property Q for a normed linear space $(X, \|\cdot\|)$ was introduced in \cite{BM}. 

\begin{defn}\label{defn:propQ}
A finite-dimensional normed linear space $X$ has \emph{Property~Q} if for all non-negative definite $B$ in $X\otimes X$, we have $\|B\|_\pi = \|B\|_\epsilon$.
\end{defn}

\subsection{Property Q for $\ell_\infty^2$ and $\ell_\infty^3$ over the complex field $\mathbb{C}$}
We note that Property Q introduced in Definition \ref{defn:propQ} depends on the ground field. We first show that both $\ell_\infty^2$ and $\ell_\infty^3$ possess Property Q over the complex field. In investigating Property~Q for $\ell_\infty^n$ over the complex field, the case $n=3$ is a threshold: $\ell_\infty^n$ has Property~Q if and only if $n \leq 3$. This is proved in \cite[Theorem~2.3]{BM}. Part of that proof relies on the fact that $\ell_\infty^3$ has Property~Q, for which the original proof in \cite{BM} is not complete. Here we give a simple and self-contained proof.

An explicit formula for the injective norm in $\ell^n_\infty \otimes \ell^n_\infty$ is easy to find. For any $A\in \ell_\infty^n\otimes \ell_\infty^n$,
\[
\|A\|_{\ell_1^n \to \ell_\infty^n} = \max\{|a_{i,j}|: 1\leq i,j \leq n\}.
\]
If $A$ is also non-negative definite, then $A= \big (\!\big (\langle a_i , a_j \rangle\big )\!\big )_{i,j=1}^n$ for some vectors $a_1, \ldots , a_n$. By the Cauchy--Schwarz inequality, $|a_{ij}|^2 \leq \|a_i\|^2 \|a_j\|^2 = a_{ii}\, a_{jj}$. It follows that
\begin{equation}\label{injnorm}
\|A\|_\epsilon= \max_{1\leq i \leq n}\{a_{ii}\}.
\end{equation}

\begin{thm}\label{thm-propQ-linf3}
Both $\ell^2_\infty$ and $\ell_\infty^3$ have Property~Q.
\end{thm}
\begin{proof} Throughout, $A=\big (\!\big ((a_{ij}\big )\!\big )$ is non-negative definite and, by \eqref{injnorm}, $\|A\|_\epsilon=\max_k a_{kk}=:M$. Since $\|A\|_\epsilon\le\|A\|_\pi$ always holds, it suffices to prove $\|A\|_\pi\le M$. We establish this for both $n=2$ and $n=3$ via two cases. In Case 1,  all diagonal entries are strictly positive ($a_{kk} > 0$), allowing us to normalize $A$ to a correlation matrix. In Case 2, at least one diagonal entry is zero, which collapses the matrix to a lower-dimensional block. Then, in particular, the $n=3$ case is either trivial to solve directly, or reduces to the case of a $2\times2$ block handled in Case 1.

\subsubsection*{Case 1: All $a_{kk}>0$} Let $D=\operatorname{diag}(d_1,\dots,d_n)$, $d_k=\sqrt{a_{kk}}>0$ ($n\in\{2,3\}$).
Then $B:=D^{-1}AD^{-1}$ is Hermitian, non-negative  definite, with $B_{kk}=1$:
an $n\times n$ correlation matrix. The set $\mathcal{C}_n$ of such matrices is a closed, bounded (non-negative definite with unit diagonal gives $|B_{ij}|\le1$), hence compact, convex subset of the \emph{real} vector space $\mathcal{H}_n$ of $n\times n$ Hermitian matrices. As the diagonal is fixed at $1$, $\mathcal{C}_n$ lies in the affine subspace determined by the off-diagonal entries, of real dimension
$2\binom{n}{2}$.

By Minkowski's theorem, $\mathcal{C}_n$ equals the convex hull of its extreme points. 
By \cite[Theorem 3]{CV}, in the complex case an extreme point of $\mathcal{C}_n$ has rank $r$ with $r^2\le n$; for $n\in\{2,3\}$ this forces $r=1$. (For a full characterization of extreme correlation matrices for arbitrary $n$, see \cite{LT}.) A rank-one Hermitian non-negative definite matrix with unit diagonal is necessarily of the form $u\otimes\bar u$ with $|u_i|=1$, so every extreme point of $\mathcal{C}_n$ has this form.
Applying Carathéodory's theorem in the affine hull of $\mathcal{C}_n$, which has
real dimension $2\binom{n}{2}$, we can write
\[
  B \;=\; \sum_{i=1}^{m} \lambda_i \, \bigl(u^{(i)} \otimes \overline{u^{(i)}}\bigr),
  \qquad \lambda_i \geq 0, \;\; \sum_{i=1}^m \lambda_i = 1,
\]
with $m \leq 2\binom{n}{2} + 1$; concretely, $m \leq 3$ when $n = 2$ and
$m \leq 7$ when $n = 3$. In either case the sum is finite, which is all we use below. Since $D$ is real, $D\overline{u}=\overline{Du}$, so
\[
A=DBD=\sum_{i=1}^m \lambda_i\,(Du^{(i)})\otimes\overline{(Du^{(i)})}.
\]
Set $w^{(i)}:=Du^{(i)}$; then $|w^{(i)}_k|=\sqrt{a_{kk}}$. Consequently, 
$\|w^{(i)}\|_{\ell_\infty}^2=\max_k a_{kk}=M$. Therefore, using
$\|x\otimes\bar x\|_\pi=\|x\|_{\ell_\infty}^2$ and the triangle inequality on
the \emph{finite} sum we deduce that 
\[
\|A\|_\pi\le\sum_{i=1}^m\lambda_i\|w^{(i)}\|_{\ell_\infty}^2
=\sum_{i=1}^m\lambda_i M=M.
\]
Hence $\|A\|_\pi=M=\|A\|_\epsilon$.

\subsubsection*{Case 2: $a_{kk}=0$ for some $k$.}
Since $A \geqslant 0$, the minor $\left (\begin{smallmatrix}0&a_{kj}\\
\overline{a_{kj}}&a_{jj}\end{smallmatrix}\right )$ has determinant $-|a_{kj}|^2\ge0$,
so $a_{kj}=0$ for all $j$: the $k$th row and column vanish. Let
$S=\{k:a_{kk}>0\}$. 
\begin{enumerate}
\item If $S=\emptyset$ then $A=0$ and $\|A\|_\pi = 0 =\|A\|_\epsilon$.
\item If $|S|=1$, say $S=\{j\}$, then $A=a_{jj}\,e_j\otimes\bar e_j$ is an elementary tensor and $\|A\|_\pi=\|A\|_\epsilon=a_{jj}=M$.
\item If $n=3$ and $|S|=2$, then $A$ is supported on the $2\times2$ block indexed by $S$. The canonical coordinate projection operator $P_S : \ell_\infty^3 \to \ell_\infty^2$ has operator norm $1$, which ensures that $\|A\|_\pi$ and $\|A\|_\epsilon$ are equal to the respective norms of this $2\times2$ block, and the $\ell_\infty^2$ case (Case 1, $n=2$) applies.
\end{enumerate}
This proves Property~Q for both $\ell_\infty^2$ and $\ell_\infty^3$.
\end{proof}
\begin{rem}\label{rem:Q-extreme-points}
The proof of Theorem~\ref{thm-propQ-linf3} given above is self-contained; we record the alternative route that motivated the extreme-point analysis of Section~2. By the duality $(X \otimes_\pi X)^*= X^* \otimes_\epsilon X^*$ (see \cite[\S 6.4]{DF}), Property~Q for $\ell^3_\infty$ states that $\|B\|_\pi = \|B\|_\epsilon$ for all $B \ge 0$, where$$\|B\|_\pi = \sup \big\{\, |\langle T,B\rangle| : T \text{ a contraction } \ell^3_\infty\to\ell^3_1 \,\big\}.$$Since $B \mapsto \|B\|_\pi$ is convex, verifying that $\|B\|_\pi \le 1$ on $\mathcal A^{(3)}_{1\to\infty}(\mathbb C)$ reduces to evaluating this norm on its extreme points. By Corollary~\ref{cor-complex-n3}, every non-zero extreme point of $\mathcal A^{(3)}_{1\to\infty}(\mathbb C)$ is a rank-one matrix $B=v\otimes\bar v$ with $\max_i|v_i|^2=1$, yielding $\|B\|_\pi=\|v\|_\infty^2=1=\|B\|_\epsilon$. Thus, $\|B\|_\pi \le \|B\|_\epsilon$ throughout, which is Property~Q for $\ell^3_\infty$.
\end{rem}

\subsection{Property Q for  $\ell_\infty^2$ and $\ell_\infty^3$ over the real field $\mathbb{R}$}\label{sec:propertyQ-constant}
We have just proved that both  $\ell_\infty^2$ and $\ell_\infty^3$ have Property Q over the complex field. Over the real field, while $\ell_\infty^2$ has Property Q, $\ell_\infty^3$ does not. In this subsection, the scalar field is assumed to be real. First, we quantify Property Q by setting  $\rho^+(X)$ to be the smallest constant such that
\begin{equation} \label{constQ}
\|A\|_\pi \leq \rho^+(X)\|A\|_\epsilon \text { for all } A \geq 0
\end{equation}
In the special case, when $\rho^+(X)=1$, the normed linear space $X$ has Property Q matching with Definition \ref{defn:propQ}. 

\begin{thm}\label{thm-Q}
$\rho^+(\ell_\infty^2)=1$ and $\rho^+(\ell_\infty^3)=\frac{5}{4}$.
\end{thm}
\begin{proof} 
By the duality recalled in Remark \ref{rem:Q-extreme-points} we can express the parameter $\rho^{+}\left(\ell_{\infty}^n\right)$ as the following supremum:
\[
\rho^{+}(\ell_\infty^3)
= \sup \left\{ |\langle T, B \rangle| : \|T\|_{\ell_\infty^3 \to \ell_1^3} \leq 1,\ \|B\|_{\ell_1^3 \to \ell_\infty^3} \leq 1,\ B \geq 0 \right\}.
\]
For $n=2$, every extreme point of the unit ball of $B\left(\ell_{\infty}^2, \ell_1^2\right)$ is of the form $D_1 P_1 E P_2 D_2$, where $P_1, P_2$ are permutation matrices, $D_1, D_2$ are diagonal matrices with unimodular entries $( \pm 1)$, and the matrix $E$ is one of the following base matrices (see \cite{Lima, BaGm}):
\[
E_1 =
\begin{pmatrix}
1 & 0 \\
0 & 0
\end{pmatrix},
\qquad
E_2 =
\begin{pmatrix}
\frac{1}{2} & \frac{1}{2}\\
\frac{1}{2} & -\frac{1}{2}\\
\end{pmatrix}.
\]
Substituting the Gram matrix representation of $B=\big (\!\big ( \langle v_j,v_i \rangle \big)\!\big )_{i,j=1}^2$ and evaluating $|\langle T, B\rangle|$ over these extreme points, we obtain two distinct types of optimization problems:

From $E_1$, we get $\sup _{\left\|v_1\right\|_2 \leq 1}\left\|v_1\right\|^2=1$. From $E_2$, we get 
\begin{align*}
\sup _{\left\|v_i\right\|_2 \leq 1} \frac{1}{2}\left|\left\|v_1\right\|^2+2\left\langle v_1, v_2\right\rangle-\left\|v_2\right\|^2\right| &= \frac{1}{2} \sup _{0 \leq r_i \leq 1, \theta_i \in \mathbb{R}}\left|r_1^2+2 r_1 r_2 \cos \left(\theta_1-\theta_2\right)-r_2^2\right|.
\end{align*}
If the maximum is achieved when the term inside the modulus is non-negative, then  $\cos \left(\theta_1-\theta_2\right)$ must be chosen to be $1$ to maximize the expression. Since the function $\frac{1}{2} \sup _{0 \leq r_i \leq 1}(r_1^2+2 r_1 r_2-r_2^2)$
is strictly increasing with respect to $r_1$, the maximum must occur at the boundary $r_1 = 1$. This reduces our problem to:
$\frac{1}{2} \sup _{0 \leq r_2 \leq 1}\left|1+2 r_2-r_2^2\right|$. The quadratic function $1+2 r_2-r_2^2$ achieves its maximum value of 2 at $r_2=1$ proving $\rho^{+}\left(\ell_{\infty}^2\right)=1$.

For $n=3$, every extreme point of the unit ball of $B(\ell_\infty^3, \ell_1^3)$ takes the form $D_1 P_1 E P_2 D_2$, see \cite{Lima, BaGm}, where the base matrix $E$ is now given by
\[
E_1 =
\begin{pmatrix}
1 & 0 & 0 \\
0 & 0 & 0 \\
0 & 0 & 0
\end{pmatrix},
\qquad
E_2 =
\begin{pmatrix}
\frac{1}{2} & \frac{1}{2} & 0 \\
\frac{1}{2} & -\frac{1}{2} & 0 \\
0 & 0 & 0
\end{pmatrix}.
\]
In this case the Gram matrix representation is $B=\big (\!\big ( v_j,v_i\big)\!\big )_{i,j=1}^3$. The extreme points generated by $E_1$ reduce the pairing to $\left|\left\langle v_i, v_j\right\rangle\right|$, which is bounded by $1$ via the Cauchy-Schwarz inequality; since this is strictly dominated by the $E_2$ cases, they are omitted from the list below. Testing across all combinations of admissible extreme points generated by $E_2$ leads to the following exhaustive list of extremization problems:
\begin{itemize}
\item[(a)] $\sup_{\|v_i\|_2\leq 1}\frac{1}{2}|\|v_1\|^2+2\langle v_1,v_2\rangle-\|v_2\|^2|$.
\item[(b)]  $\sup_{\|v_i\|_2\leq 1}\frac{1}{2}|\|v_1\|^2+\langle v_1,v_3\rangle+\langle v_2,v_1\rangle-\langle v_2,v_3\rangle|$.
\item[(c)]  $\sup_{\|v_i\|_2\leq 1}\frac{1}{2}|\|v_1\|^2+\|v_2\|^2|$.
\item[(d)]  $\sup_{\|v_i\|_2\leq 1}\frac{1}{2}|\langle v_1,v_2\rangle+\langle v_1,v_3\rangle+\|v_2\|^2-\langle v_2,v_3\rangle|$.
\item[(e)]  $\sup_{\|v_i\|_2\leq 1}\frac{1}{2}|\langle v_1,v_3\rangle+\|v_1\|^2+\langle v_2,v_3\rangle-\langle v_2,v_1\rangle|$.
\item[(f)]  $\sup_{\|v_i\|_2\leq 1}\frac{1}{2}|\langle v_1,v_3\rangle+\langle v_1,v_2\rangle+\langle v_2,v_3\rangle-\|v_2\|^2|$.
\end{itemize}
A reduction of this list is possible via symmetries as follows. 
\begin{enumerate}
\item Cases (a) and (c) are inherited directly from the $n=2$ subsystem, meaning their maxima are bounded by $1$.
\item By replacing $v_1$ with $-v_1$, we see that (d) and (f) are equivalent.
\item Swapping the roles of $v_1$ and $v_2$ in (e) with $-v_2$ and $-v_1$ reveals that (e) is equivalent to (f).
\item Replacing $v_1$ with $-v_1$ in (b) and permuting indices shows that (b) is also equivalent to (f).
\end{enumerate}
Thus, all remaining non-trivial cases collapse into a single optimization problem, which we choose 
to evaluate using format (e):
\[\sup _{\left\|v_i\right\|_2 \leq 1} \frac{1}{2}\left|\left\langle v_1, v_3\right\rangle+\left\|v_1\right\|^2+\left\langle v_2, v_3\right\rangle-\left\langle v_2, v_1\right\rangle\right|.\]
Rewriting the expression $\langle v_1,v_3\rangle+\|v_1\|^2+\langle v_2,v_3\rangle-\langle v_2,v_1\rangle$, we get $\langle v_1+v_2,\, v_3\rangle + \langle v_1-v_2,\, v_1\rangle.$
Maximizing over $v_3,$ the problem reduces to:
\[
\sup_{\|v_1\|,\|v_2\|\leq 1}\ \frac{1}{2}\Big(\|v_1+v_2\| + |\langle v_1-v_2,\, v_1\rangle|\Big).
\]
Without loss of generality, we may set $v_1 = (\alpha,0)$ with $\alpha \in [0,1],$ and $v_2 = (s, t)$ with $s^2+t^2 \leq 1$. Then, we need to maximize the following function:
\[
f(\alpha,s,t) = \frac{1}{2}\left(\sqrt{(\alpha+s)^2+t^2}\ +\ \alpha(\alpha-s)\right).
\]
Since $f$ is increasing in $\alpha,$ therefore $\alpha=1$ maximizes $f$. Hence we need to solve:
\[\max_{s\in [-1,1]}\frac{1}{2}\left(\sqrt{2+2s} + 1 - s\right).\]
As can easily be seen through computation, the above maximum is achieved at $s=-\frac{1}{2}$ and is equal to $5/4.$
\end{proof}

\subsection{Property Q for $(\mathbb{C}^2, \|\cdot\|_{\boldsymbol{A}}$)}\label{sec:propertyQ-normA} 
Fix an $m$-tuple $\boldsymbol{A}:=(A_1, \ldots, A_m)$ of $d\times d$ linearly independent matrices. Define a norm on $\mathbb{C}^m$ by
\begin{equation}
\|z\|_{\boldsymbol{A}} := \|z_1 A_1 + \cdots + z_m A_m \|_{\ell^d_2 \to \ell^d_2}.
\end{equation}
If $m=d^2$, then the mapping $\iota$, defined by $\iota(z_1, \ldots , z_m) = z_1 A_1 + \cdots + z_m A_m,$ is an isometric isomorphism, identifying $(\mathbb{C}^m, \|\cdot\|_{\boldsymbol{A}})$ with $\mathbb{C}^{d\times d}$ equipped with the standard operator norm.
Assume $m< d^2$, and let $X_{\boldsymbol{A}}$ be the $m$-dimensional subspace of $\mathbb{C}^{d\times d}$ spanned by $\{A_1, \ldots, A_m\}$. Then, the map $\iota$ acts as an isometric isomorphism from $(\mathbb{C}^m, \|\cdot\|_{\boldsymbol{A}})$ onto $X_{\boldsymbol{A}}$ equipped with the operator norm.

If $\ell:X_{\boldsymbol{A}} \to \mathbb{C}$ is a linear functional, then it has an extension $\hat{\ell}$ to $(\mathbb{C}^{d\times d}, \|\cdot\|_{\ell_2^d \to \ell_2^d})$ with $\|\ell\| = \|\hat{\ell}\|$. Then
\[
\ell(\iota (z_1, \ldots , z_m)) = z_1 \operatorname{tr} (A_1C^*) + \cdots + z_m \operatorname{tr} (A_mC^*)
\]
for some $C\in \mathbb{C}^{d\times d}$. Therefore, the unit ball in the dual space $(\mathbb{C}^m, \|\cdot\|_{\boldsymbol{A}})^*$ is
\[
\Big\{\Big( \sum_{i,j=1}^d A_1(i,j) \overline{C(i,j)}, \ldots , \sum_{i,j=1}^d A_m(i,j) \overline{C(i,j)}\Big): \|C\|_{\operatorname{tr}} \leq 1\Big\}.
\]

Following \cite{AFJS95}, for any $a_1, a_2 \in \mathbb{C}$, we consider the norm on $\mathbb{C}^2$ induced by setting $A_1$ to be the diagonal matrix with entries $(1,0,a_1)$ and $A_2$ to be the diagonal matrix with entries $(0,1,a_2)$. This yields the norm
\[
  \|(z_1,z_2)\|_{\boldsymbol{A}} = \max\bigl\{|z_1|,\, |z_2|,\, |a_1 z_1 + a_2 z_2|\bigr\}.
\]
Up to $\ell_{\infty}^3$-isometries, any 2-dimensional subspace of $\ell_{\infty}^3$ can be parameterized by a basis of the form $\{(1,0, a_1), (0,1, a_2)\}$, and is therefore isometric to $\mathbb{C}^2$ equipped with this norm. Throughout this subsection, we fix an arbitrary pair $(a_1, a_2)$ and write
\[
  X := (\mathbb{C}^2, \|\cdot\|_{\boldsymbol{A}}),
  \qquad
  X^* := (\mathbb{C}^2, \|\cdot\|_{\boldsymbol{A}}^*),
\]
where $\|\cdot\|_{\boldsymbol{A}}^*$ denotes the dual norm. 

In \cite[Proposition 4.4]{AFJS95}, it is proved that $X$ has the 2-summing property, or equivalently, Property P of \cite{BM}. Unlike \cite{AFJS95}, which treats the isometric and non-isometric cases 
( $|a_1|+ |a_2| \leqslant 1$ or not) separately, we give a unified algebraic proof that $X$ satisfies the stronger Property Q for all choices of $(a_1, a_2)$.

\begin{lem}\label{lem:dual-ball-and-contraction}
Let $X := (\mathbb{C}^2, \|\cdot\|_{\boldsymbol{A}})$, where
$\|(z_1, z_2)\|_{\boldsymbol{A}} = \max\bigl\{|z_1|,\, |z_2|,\, |a_1 z_1 + a_2 z_2|\bigr\}$.
\begin{enumerate}
  \item[(i)] The dual norm on $X^*$ is given by
  \[
    \|(\hat\alpha_1, \hat\alpha_2)\|_{X^*}
    \;=\; \inf\Bigl\{|\alpha_1| + |\alpha_2| + |\alpha_3|
      \,:\, \hat\alpha_1 = \alpha_1 + a_1 \alpha_3,\; \hat\alpha_2 = \alpha_2 + a_2 \alpha_3\Bigr\},
  \]
  and consequently the closed unit ball of $X^*$ is
  \[
    B_{X^*}
    \;=\; \bigl\{(\alpha_1 + a_1 \alpha_3,\; \alpha_2 + a_2 \alpha_3)
      \,:\, |\alpha_1| + |\alpha_2| + |\alpha_3| \leq 1\bigr\}.
  \]
  \item[(ii)] For any pair of vectors $v_1, v_2 \in \mathbb{C}^n$, the linear map
  $V = (v_1, v_2) : X^* \to \mathbb{C}^n$ defined by
  $V(\hat\alpha_1, \hat\alpha_2) := \hat\alpha_1 v_1 + \hat\alpha_2 v_2$
  satisfies
  \[
    \|V\|_{X^* \to \mathbb{C}^n} \leq 1
    \;\iff\;
    \max\bigl\{\|v_1\|_2,\; \|v_2\|_2,\; \|a_1 v_1 + a_2 v_2\|_2\bigr\} \leq 1.
  \]
\end{enumerate}
\end{lem}

\begin{proof}
The norm on $X$ is the $\ell^\infty$ norm pulled back along the isometric embedding
\[
  \iota : X \hookrightarrow \ell^\infty(3),
  \qquad (z_1, z_2) \longmapsto (z_1,\, z_2,\, a_1 z_1 + a_2 z_2).
\]
Every contractive linear functional $\ell : \ell^\infty(3) \to \mathbb{C}$ has the form
\[
  \ell(z_1, z_2, z_3) = \alpha_1 z_1 + \alpha_2 z_2 + \alpha_3 z_3,
  \qquad |\alpha_1| + |\alpha_2| + |\alpha_3| \leq 1,
\]
so the restriction $\hat\ell := \ell \circ \iota : X \to \mathbb{C}$ is given by
\[
  \hat\ell(z_1, z_2) = \hat\alpha_1 z_1 + \hat\alpha_2 z_2,
  \qquad \hat\alpha_i := \alpha_i + a_i \alpha_3 \;\; (i = 1, 2),
\]
and is plainly a contraction. Conversely, any contractive $\hat\ell : X \to \mathbb{C}$
admits a norm-preserving extension to $\ell^\infty(3)$ by the Hahn--Banach theorem.
Therefore the unit ball $B_{X^*}$ coincides with the image of
$B_{\ell^\infty(3)^*} = B_{\ell^1(3)}$ under
$(\alpha_1, \alpha_2, \alpha_3) \mapsto (\alpha_1 + a_1 \alpha_3,\, \alpha_2 + a_2 \alpha_3)$,
which establishes both formulae in~(i).

For~(ii), using (i),
\[
  V(\hat\alpha_1, \hat\alpha_2)
  \;=\; \hat\alpha_1 v_1 + \hat\alpha_2 v_2
  \;=\; \alpha_1 v_1 + \alpha_2 v_2 + \alpha_3 \bigl(a_1 v_1 + a_2 v_2\bigr),
\]
so $\|V\|_{X^* \to \mathbb{C}^n}$ equals the operator norm of the map
$\ell^1(3) \to \mathbb{C}^n$ whose columns are the three vectors
$v_1,\, v_2,\, a_1 v_1 + a_2 v_2$. The norm of such a map is the maximum of the
$\ell^2$ norms of its columns, giving the claimed equivalence.
\end{proof}

\begin{thm}\label{thm-propQ-normA}
The normed linear space $(\mathbb{C}^2, \|\cdot\|_{\boldsymbol{A}})$ has Property~Q.
\end{thm}

For the proof, we need a couple of preparatory lemmas. 

\begin{lem}
The set $\mathscr{G} = \{0\leq V\in X\otimes X : \|V\|_\epsilon \leq 1\}$ is characterized as
\[
\mathscr{G} = \left\{V = \begin{pmatrix} x & z \\ \bar{z} & y \end{pmatrix} \geq 0 : x \leq 1,\; y \leq 1,\; |a_1|^2 x + |a_2|^2 y + 2\Re(a_1 \bar{a}_2 z) \leq 1 \right\}.
\]
\end{lem}

\begin{proof} 
A non-negative definite $V \in X \otimes X$ admits a Cholesky factorization
\[
  V = U^* U,
  \qquad U = [\,v_1 \mid v_2\,] : \mathbb{C}^2 \to \mathbb{C}^n,
\]
so that $V_{ij} = \langle v_j, v_i\rangle$; in particular,
$x = \|v_1\|_2^2$, $y = \|v_2\|_2^2$, and $z = \langle v_2, v_1\rangle$.

Under the canonical identification of $X \otimes X$ with the space of linear maps
$X^* \to X$ (valid in finite dimensions), the injective norm coincides with the
operator norm, and for $V = U^* U$ this gives
\[
  \|V\|_\epsilon \;=\; \|U\|_{X^* \to \mathbb{C}^n}^{\,2}.
\]
By Lemma~\ref{lem:dual-ball-and-contraction}(ii), the contraction condition
$\|U\|_{X^* \to \mathbb{C}^n} \leq 1$ is equivalent to
\[
  \|v_1\|_2 \leq 1, \qquad \|v_2\|_2 \leq 1, \qquad \|a_1 v_1 + a_2 v_2\|_2 \leq 1.
\]
Squaring these and expanding
\[
  \|a_1 v_1 + a_2 v_2\|_2^{\,2}
  \;=\; |a_1|^2 \|v_1\|_2^{\,2} + |a_2|^2 \|v_2\|_2^{\,2}
        + 2\,\Re\bigl(a_1 \bar{a}_2 \langle v_2, v_1\rangle\bigr)
\]
yields the three conditions in the statement.
\end{proof}

\begin{lem}\label{lem:rank-one}
Every extreme point of $\mathscr{G}$ has rank at most one.
\end{lem}
\begin{proof} Assume for a contradiction that there exists an extreme point $V \in \mathscr{G}$ with $\operatorname{rank}(V) = 2$. Since $V$ is a $2 \times 2$ non-negative definite matrix and has rank $2$, it must be strictly positive definite. Thus, $\det(V) = xy - |z|^2 > 0$. By the continuity of the determinant, there exists an $\epsilon > 0$ such that $V \pm \Delta$ remains strictly positive definite for any Hermitian perturbation $\Delta$ satisfying $\|\Delta\| < \epsilon$.

Consider an off-diagonal perturbation of the form
$$ \Delta = \begin{pmatrix} 0 & \delta \\ \bar{\delta} & 0 \end{pmatrix} $$
where $\delta \in \mathbb{C} \setminus \{0\}$ is to be determined such that $\|\Delta\| < \epsilon$. For the perturbed matrices $V_\pm = V \pm \Delta$, the first two constraints characterizing $\mathscr{G}$ ($x \leq 1$ and $y \leq 1$) are automatically preserved since the diagonal entries remain strictly unchanged.

For the third constraint, evaluating at $V_\pm$ yields:
$$ |a_1|^2 x + |a_2|^2 y + 2\Re\bigl(a_1 \bar{a}_2 (z \pm \delta)\bigr) = |a_1|^2 x + |a_2|^2 y + 2\Re(a_1 \bar{a}_2 z) \pm 2\Re(a_1 \bar{a}_2 \delta). $$

Since $V \in \mathscr{G}$, we know $|a_1|^2 x + |a_2|^2 y + 2\Re(a_1 \bar{a}_2 z) \leq 1$. We split the analysis into two cases based on whether this constraint is strict.

\textit{Case A}: Suppose $|a_1|^2 x + |a_2|^2 y + 2\Re(a_1 \bar{a}_2 z) < 1$. We can choose any $\delta \neq 0$ with sufficiently small magnitude $|\delta| < \epsilon$ such that the $\pm 2\Re(a_1 \bar{a}_2 \delta)$ term does not violate the strict inequality. Thus, $V_\pm \in \mathscr{G}$.

\textit{Case B}: Suppose $|a_1|^2 x + |a_2|^2 y + 2\Re(a_1 \bar{a}_2 z) = 1$. Let $\mu := a_1 \bar{a}_2$. We require $\Re(\mu \delta) = 0$ to prevent violating the equality.
\begin{itemize}
  \item If $\mu = 0$, the third constraint for $V_\pm$ reduces to $|a_1|^2 x + |a_2|^2 y = 1$, which is independent of $\delta$. We may choose any $\delta \neq 0$ with $|\delta| < \epsilon$, and $V_\pm \in \mathscr{G}$ holds.
  \item If $\mu \neq 0$, we choose $\delta = i t \bar{\mu}$ for some real number $t$ chosen such that $0 < t < \epsilon / |\mu|$. Then $\mu \delta = i t |\mu|^2$ is purely imaginary, ensuring $\Re(\mu \delta) = 0$. Hence, the third constraint is preserved, and $V_\pm \in \mathscr{G}$.
\end{itemize}

In all cases, we have found a non-zero perturbation $\Delta$ such that $V_+ \in \mathscr{G}$ and $V_- \in \mathscr{G}$. Since $V = \frac{1}{2}(V_+ + V_-)$, this contradicts the assumption that $V$ is an extreme point of $\mathscr{G}$. Therefore, every extreme point of $\mathscr{G}$ must have rank at most one.
\end{proof}

\begin{proof}[Proof of Theorem~\ref{thm-propQ-normA}]
For every $V \geq 0$ with $\|V\|_\epsilon \leq 1$ we must show that $\|V\|_\pi \leq 1$.
As before, by the injective-projective duality, we have
\[\|V\|_\pi = \sup\bigl\{|\langle P,V\rangle|: P\in X^*\otimes X^*,\ \|P\|_\epsilon\le1 \bigr\}.
\]
Fix $P \in X^* \otimes X^*$ with $\|P\|_\epsilon \leq 1$. Since $\mathscr{G}$ is compact and convex, 
the continuous functional $V\mapsto |\langle P,V\rangle|$ attains its maximum on
$\operatorname{Ext}(\mathscr{G})$. By Lemma~\ref{lem:rank-one},
every extreme point of $\mathscr G$ is of the form
$V=v\otimes\bar v$. Hence
\[|\langle P,V\rangle|\leqslant \|P\|_\epsilon\,\|V\|_\pi = \|P\|_\epsilon\,\|V\|_\epsilon \leqslant 1,
\,\, V\in\operatorname{Ext}(\mathscr G),
\]
since the injective and projective norms agree on elementary tensors.
Therefore, 
\[
\sup_{V\in\mathscr{G}}|\langle P,V\rangle| \leqslant 1.
\]
Taking the supremum over all $P$ with $\|P\|_\epsilon\leqslant 1$, we obtain
\[
\|V\|_\pi\leqslant 1\,\,
V\in\mathscr{G}.
\]
Since $\|V\|_\epsilon\le\|V\|_\pi$, it follows that $\|V\|_\pi=\|V\|_\epsilon$
for every $V\geqslant 0$.
\end{proof}

As noted at the beginning of this subsection, every 2-dimensional subspace of $\ell_{\infty}^3$ is isometric to $(\mathbb{C}^2, \|\cdot\|_{\boldsymbol{A}})$ for some choice of $a_1, a_2 \in \mathbb{C}$. Since Theorem~\ref{thm-propQ-normA} establishes Property~Q for all such pairs unconditionally, the following result strengthens \cite[Proposition 4.4]{AFJS95} and follows immediately.

\begin{cor}\label{cor:2dimPropQ}
Every 2-dimensional subspace of $\ell_{\infty}^3$ has Property Q.
\end{cor}
\begin{rem}
   The diagonal hypothesis on $(A_1, A_2)$ is essential. If instead $A_1, A_2$ are $2 \times 2$ matrices that are not simultaneously diagonalizable, then $(\mathbb{C}^2,\|\cdot\|_{\mathbf{A}})$  fails Property~P, and a fortiori Property~Q as shown in \cite{MPV}.
\end{rem}
\appendix
\section{More on Property~Q}
It is natural to view Property Q from the perspective of \emph{norming sets}. 
Recall that the projective tensor norm on $X \otimes X$ 
is given by the dual formula (by the duality $(X\otimes_\pi X)^* = X^* \otimes_\epsilon X^*$, see \cite[\S 6.4]{DF})
$$
\|B\|_\pi = \sup\{ |\langle A,B\rangle| : \|A\|_\epsilon \le 1 \},
$$
where $A$ ranges over $X^*\otimes_\epsilon X^*$. Thus the unit ball of 
$X^*\otimes_\epsilon X^*$ serves as a \emph{norming set} for the projective tensor product.

Property Q provides a simplification of this dual formula. If a space possesses Property Q, it implies 
that for any non-negative definite element $B \geqslant 0$, one can restrict the supremum to the 
\emph{non-negative} part of the injective unit ball up to a constant determined by the conjugation, 
with exact restriction when the conjugation is isometric.

Let $X$ be a finite-dimensional normed linear space over $\mathbb C$ with a distinguished basis 
$e_1,\dots,e_n$, and let $e_1^{*},\dots,e_n^{*}$ be the dual basis of $X^{*}$. For $x\in X$, 
let $\bar x$ denote the coordinatewise conjugate vector, and for $\mu\in X^{*}$, let $\bar\mu\in X^{*}$ 
be the functional with conjugated coefficients, $\bar\mu(e_k):=\overline{\mu(e_k)}$, 
or equivalently $\bar\mu(x)=\overline{\mu(\bar x)}$. An element $B=\sum_{j,k}B_{jk}\,e_j\otimes e_k$ 
of $X\otimes X$ is \emph{non-negative definite}, written $B\geqslant 0$, if its coefficient 
matrix $\big(B_{jk}\big)$ is non-negative definite. For $X^{*}\otimes X^{*}$, this non-negativity is equivalent to a representation $A=\sum_{s}\mu_s\otimes\bar\mu_s$ with $\mu_s\in X^{*}$. This is the cone of non-negative elements used in \cite{BM}: the representation $A=\sum_s\mu_s\otimes\bar\mu_s$ is the factorization
$A=B^{*}B$ of \cite{BM}, and it agrees with the description there as the
convex hull of the tensors $x\otimes\bar x$. (Over $\mathbb C$ the latter
should be read with the conjugate in the second factor; the un-conjugated
symmetric tensors $x\otimes x$ give the same cone precisely when the
conjugation is isometric.) The positive cone of contractions is defined as 
\[
\mathcal{C}=\{\,A\in X^{*}\otimes X^{*} : A\geqslant 0,\ \|A\|_\epsilon\leqslant 1\,\}.
\]

The map $\mu\mapsto\bar\mu$ is a (real-linear) involution of $X^{*}$, but it need not 
be an isometry. Set
\[
\kappa \;:=\; \sup\{\,\|\bar\mu\|_{X^{*}} : \mu\in X^{*},\
\|\mu\|_{X^{*}}\leqslant 1\,\}.
\]
Since $\|\mu\|=\|\bar{\bar\mu}\|\leqslant\kappa\|\bar\mu\|\leqslant \kappa^{2}\|\mu\|$, 
we have $1\leqslant\kappa<\infty$, with $\kappa=1$ if and only if $\mu\mapsto\bar\mu$ 
is an isometry of $X^{*}$. The same constant results if the supremum is computed in $X$. Indeed $\|\bar\mu\|_{X^{*}}=\sup\{|\mu(\bar x)|:\|x\|_X\leqslant1\}$, whence 
$\kappa=\sup\{\|\bar x\|_X:\|x\|_X\leqslant1\}$.

\begin{prop}[Quantitative positive norming estimate]\label{prop:positive-norming-kappa}
Assume that $X$ has Property~Q. Then, for every non-negative definite
$B\in X\otimes X$,
\[
\sup_{A\in\mathcal C}\langle A,B\rangle
\;\leqslant\;
\|B\|_\pi
\;\leqslant\;
\kappa\,\sup_{A\in\mathcal C}\langle A,B\rangle .
\]
In particular, if $\kappa=1$, then $\|B\|_\pi=\sup_{A\in\mathcal C}\langle A,B\rangle$.
\end{prop}

\begin{proof}
Since $B$ is non-negative definite, write
$B=\sum_{j,k}\langle b_j,b_k\rangle\,e_j\otimes e_k$ for vectors $b_j$ in
an auxiliary Hilbert space, and for $\lambda\in X^{*}$ set
$\zeta_\lambda:=\sum_{j=1}^{n}\lambda(e_j)\,b_j$. For all
$\lambda,\mu\in X^{*}$,
\begin{equation}\label{eq:bform}
B(\lambda,\mu)
=\sum_{j,k}\langle b_j,b_k\rangle\,\lambda(e_j)\,\mu(e_k)
=\big\langle \zeta_\lambda,\ \zeta_{\bar\mu}\big\rangle .
\end{equation}

For the first inequality, let $A\in\mathcal C$ and write
$A=\sum_s\mu_s\otimes\bar\mu_s$. By \eqref{eq:bform}, and since
$\bar{\bar\mu}_s=\mu_s$,
\[
\langle A,B\rangle
=\sum_s B(\mu_s,\bar\mu_s)
=\sum_s\langle\zeta_{\mu_s},\zeta_{\mu_s}\rangle
=\sum_s\|\zeta_{\mu_s}\|^{2}\;\geqslant\;0 .
\]
Since $\|A\|_\epsilon\leqslant1$, the duality gives
$\langle A,B\rangle\leqslant\|B\|_\pi$; taking the supremum over
$A\in\mathcal C$ yields the first inequality. (Property~Q is not needed for
this half.)

For the second inequality, Property~Q gives $\|B\|_\pi=\|B\|_\epsilon$, so
it suffices to prove
$\|B\|_\epsilon\leqslant\kappa\,S$, where
$S:=\sup_{A\in\mathcal C}\langle A,B\rangle$. For every non-zero
$\nu\in X^{*}$, the normalized functional
$\hat\nu:=\big(\|\nu\|\,\|\bar\nu\|\big)^{-1/2}\,\nu$ satisfies
$\overline{\hat\nu}=\big(\|\nu\|\,\|\bar\nu\|\big)^{-1/2}\,\bar\nu$, so
that $\hat\nu\otimes\overline{\hat\nu}\geqslant0$ and
$\|\hat\nu\otimes\overline{\hat\nu}\|_\epsilon
 =\|\hat\nu\|\,\|\overline{\hat\nu}\|=1$. Hence
$\hat\nu\otimes\overline{\hat\nu}\in\mathcal C$, and by \eqref{eq:bform},
\begin{equation}\label{eq:rankone}
S\;\geqslant\;\langle \hat\nu\otimes\overline{\hat\nu},\,B\rangle =\|\zeta_{\hat\nu}\|^{2}
=\frac{\|\zeta_\nu\|^{2}}{\|\nu\|\,\|\bar\nu\|}\,,
\,\, \nu\in X^{*}\setminus\{0\}.
\end{equation}
Now fix $\lambda,\mu\in X^{*}$ with
$\|\lambda\|_{X^{*}}\leqslant1$, $\|\mu\|_{X^{*}}\leqslant1$. Applying
\eqref{eq:rankone} with $\nu=\lambda$, and noting
$\|\lambda\|\,\|\bar\lambda\|\leqslant\kappa\|\lambda\|^{2}
 \leqslant\kappa$, we obtain
$\|\zeta_\lambda\|^{2}\leqslant\kappa\,S$; applying it with $\nu=\bar\mu$,
and noting $\|\bar\mu\|\,\|\bar{\bar\mu}\|=\|\bar\mu\|\,\|\mu\|
 \leqslant\kappa\|\mu\|^{2}\leqslant\kappa$, we obtain
$\|\zeta_{\bar\mu}\|^{2}\leqslant\kappa\,S$. By \eqref{eq:bform} and
Cauchy--Schwarz,
\[
|B(\lambda,\mu)| =|\langle\zeta_\lambda,\zeta_{\bar\mu}\rangle|
\leqslant\|\zeta_\lambda\|\,\|\zeta_{\bar\mu}\|\leqslant\kappa\,S .
\]
Taking the supremum over the unit balls gives $\|B\|_\epsilon\leqslant\kappa\,S$, 
as required. If $\kappa=1$, the two inequalities of the Proposition coincide.
\end{proof}

\begin{rem} The constant $\kappa$ stems from formulating positivity in $X\otimes X$. 
The conjugate space $\bar X$ is the vector space having the same underlying set and additive structure as $X$, 
but with scalar multiplication defined by
\[
\lambda\cdot\bar x:=\overline{\bar\lambda\,x},\,\, \lambda\in\mathbb C,\ x\in X.
\]
The norm is given by $\|\bar x\|_{\bar X}:=\|x\|_X$.
Thus, the constant $\kappa$ measures the distortion introduced by identifying $X$ with its conjugate space 
$\bar{X}$. The conjugation enters the proof of Proposition~\ref{prop:positive-norming-kappa} at exactly one point, 
the normalization $\|\nu\|\,\|\bar\nu\|$ in \eqref{eq:rankone}; computed in $X^{*}\otimes_\epsilon\overline{X^{*}}$, 
where $\|\bar\nu\|_{\overline{X^{*}}}=\|\nu\|_{X^{*}}$ by definition, this normalization is $\|\nu\|^{2}\leqslant1$, 
and the estimate \eqref{eq:rankone} holds with constant one. Since Property~Q has always been stated for $X\otimes X$, we do not pursue the formulation over $X\otimes\bar{X}$ here; a treatment of the tensor-norm theory 
of $X\otimes\bar{X}$ in its own right, including the appropriate analogue of Property~Q, is deferred.
\end{rem}

\begin{rem}
Since $\mathcal C$ is a compact convex subset of the finite-dimensional
space $X^{*}\otimes_\epsilon X^{*}$, the supremum of the linear functional
$A\mapsto\langle A,B\rangle$ over $\mathcal C$ is attained at an extreme
point of $\mathcal C$, so that
\[
\sup_{A\in\mathcal C}\langle A,B\rangle
=\sup_{A\in\operatorname{Ext}(\mathcal C)}\langle A,B\rangle .
\]
The rank-one elements of $\mathcal C$ are the tensors
$\lambda\otimes\bar\lambda$ with
$\|\lambda\|_{X^{*}}\,\|\bar\lambda\|_{X^{*}}\leqslant1$, and the proof of
Proposition~\ref{prop:positive-norming-kappa} shows that, when $\kappa=1$,
these already compute the injective norm: for every non-negative definite
$B\in X\otimes X$,
\begin{equation}\label{eq:rankone-sup}
\sup_{\|\lambda\|_{X^{*}}\leqslant1}\langle\lambda\otimes\bar\lambda,\,B\rangle=\|B\|_\epsilon .
\end{equation}
Indeed,
$\langle\lambda\otimes\bar\lambda,B\rangle=B(\lambda,\bar\lambda)
\leqslant\|B\|_\epsilon\,\|\lambda\|\,\|\bar\lambda\|
\leqslant\|B\|_\epsilon$, which yields
\[
\sup_{\|\lambda\|_{X^*}\leqslant1} \langle\lambda\otimes\bar\lambda,B\rangle \leqslant
\|B\|_\epsilon.
\]
For the reverse inequality, when $\kappa=1$, the map $\mu\mapsto\bar\mu$ preserves the unit ball 
of $X^*$, and
\[
\|B\|_\epsilon = \sup_{\|\lambda\|,\|\mu\|\leqslant1} |\langle\zeta_\lambda,\zeta_{\bar\mu}\rangle|.
\]
By Cauchy--Schwarz,
\[
|\langle\zeta_\lambda,\zeta_{\bar\mu}\rangle| \leqslant \|\zeta_\lambda\|\,\|\zeta_{\bar\mu}\|
\leqslant \max\{\|\zeta_\lambda\|^2,\|\zeta_{\bar\mu}\|^2\},
\]
and the choice $\mu=\bar\lambda$, admissible since $\|\bar\lambda\|=\|\lambda\|$, attains $\|\zeta_\lambda\|^2$. Hence
\[
\|B\|_\epsilon \leqslant \sup_{\|\lambda\|_{X^*}\leqslant1} \langle\lambda\otimes\bar\lambda,B\rangle.
\]
Combining the two inequalities yields \eqref{eq:rankone-sup}.

For $\kappa>1$, the constant in Proposition~\ref{prop:positive-norming-kappa} arises from the
normalization of the rank-one tensors $\lambda\otimes\bar\lambda$: membership of $\lambda\otimes\bar\lambda$
in $\mathcal C$ requires $\|\lambda\|\,\|\bar\lambda\|\leqslant1$. Thus the factor $\kappa$ reflects the 
behaviour of these rank-one positive contractions rather than the passage from $\mathcal C$ to its extreme 
points. Whether higher-rank elements of $\mathcal C$ can remove the factor $\kappa$ remains open.
\end{rem}
Property Q is strictly stronger than requiring C to be a norming set. To quantify this gap, we recall Property P from \cite{BM}.
This property is naturally expressed in terms of the constant $\gamma^{+}(X)$, defined as the smallest constant such that
\[
\langle A,B\rangle \;\leqslant\; \gamma^{+}(X)\, \|A\|_{X^{*}\otimes_\epsilon X^{*}}\,\|B\|_{X\otimes_\epsilon X}
\]
for all non-negative definite $A\in X^{*}\otimes X^{*}$ and $B\in X\otimes X$.
In \cite{BM} the pairing is the Hilbert--Schmidt inner
product $\langle A,B\rangle_{\mathrm{HS}}=\sum_{j,k}A_{jk}\overline{B_{jk}}$;
since a non-negative definite $B$ is Hermitian, one has
$\langle A,B\rangle_{\mathrm{HS}}=\langle A,B^{\mathrm t}\rangle$, and the
transpose preserves both non-negative definiteness and the injective norm,
so the constant $\gamma^{+}(X)$ is insensitive to which of the two pairings
is used. We work with the bilinear pairing, consistent with the dual
formula for $\|\cdot\|_\pi$. Note that the symmetry of the pairing gives
$\gamma^{+}(X)=\gamma^{+}(X^{*})$.

\begin{defn}\label{defn:propP}
A finite-dimensional normed linear space $X$ has \emph{Property~P} if
$\gamma^{+}(X)=1$, that is,
$\langle A,B\rangle\leqslant\|A\|_\epsilon\|B\|_\epsilon$ for all
non-negative definite $A\in X^{*}\otimes X^{*}$ and $B\in X\otimes X$.
\end{defn}

The next corollary shows that Property~Q may be viewed as the conjunction of
Property~P and the requirement that the positive cone of contractions norm the
non-negative definite tensors.

\begin{cor} \label{cor:Q_equivalence}
Assume that $\kappa=1$. Then Property~Q holds if and only if Property~P holds and $\mathcal C$ 
is a norming set for the non-negative definite elements.  The hypothesis $\kappa=1$ is used only 
in the forward implication; the converse holds for arbitrary $\kappa$.
\end{cor}

\begin{proof}
If Property~Q holds, then $\|B\|_\pi=\|B\|_\epsilon,\,\, B\geqslant 0$.
Hence, by Proposition \ref{prop:positive-norming-kappa},
\[
\|B\|_\pi = \sup_{A\in\mathcal C}\langle A,B\rangle, \,\, B\geqslant 0.
\]
Moreover, for $A,B\geqslant 0$,
\[
\langle A,B\rangle \leqslant \|A\|_\epsilon\,\|B\|_\pi = \|A\|_\epsilon\,\|B\|_\epsilon,
\]
and therefore Property~P holds.

Conversely, suppose that Property~P holds and that $\|B\|_\pi = \sup_{A\in\mathcal C}\langle A,B\rangle,\,\, B\geqslant 0$.
Then
\[
\|B\|_\pi = \sup_{A\in\mathcal C}\langle A,B\rangle \leqslant \|B\|_\epsilon,
\]
since $\|A\|_\epsilon\le1$ for $A\in\mathcal C$ and Property~P applies.
As always, $
\|B\|_\epsilon\le\|B\|_\pi$,
and hence
\[
\|B\|_\pi=\|B\|_\epsilon,\,\,B\geqslant 0.
\]
Thus Property~Q holds. 
\end{proof}
Property~P and the norming condition of Corollary~\ref{cor:Q_equivalence} are independent hypotheses: For
$\ell_2^n$ ($n\geqslant2$), the cone $\mathcal C$ norms the non-negative definite elements, the supremum 
$\sup_{A\in\mathcal C}\langle A,B\rangle=\operatorname{tr}(B) =\|B\|_\pi$ being attained at $A=I$, while 
Property~P fails, since $\langle I,I\rangle=n>1=\|I\|_\epsilon^2$. In particular $\ell_2^n$ does not have 
Property~Q.

\section{The missing extreme points in the proof of Fact 7 from \cite{BM}}
Let $A:= \big ( \!\! \big ( \langle v_i , v_j \rangle \big )\!\!\big)_{i,j=1}^3$, where $v_1 = e_1 = v_3$ and $v_2 = 0$.
With this choice of $v_1,v_2, v_3$, 
we have $\|A\|_{\ell_1^3 \to \ell_\infty^3} = 1$, $A$ is non-negative, and the rank of $A$ is $1$. It follows from Corollary \ref{cor-complex-n3} that $A$ is an extreme point of the set $\mathcal{A}_{1 \rightarrow \infty}^{(3)}(\mathbb{C})$ (the set $\mathcal{X}$ of \cite{BM}). In the proof of Fact 7 of \cite{BM}, such rank-1 extreme points were erroneously excluded under the assumption that all the diagonal entries of an extreme point of $\mathcal{X}$ must be $1$. 

Let $B:\ell_\infty^3 \to \ell_1^3$ be a self-adjoint linear transformation of the form: 
\[
B = \begin{pmatrix}  
r & \overline{\lambda_1} & \overline{\lambda_2} \\
\lambda_1 & s & \overline{\lambda_3} \\
\lambda_2 & \lambda_3 & t
\end{pmatrix} = \begin{pmatrix}
1 & 0 & 1 \\
0 & -1 & 0 \\
1 & 0 & 1
\end{pmatrix},
\]
where we have chosen $r=1=t$, $s=-1$, $\lambda_1 = 0 = \lambda_3$, and $\lambda_2=1$. With this choice, we have  
\begin{align*}
|\langle A , B \rangle|  &= |r \|v_1\|^2 + s \|v_2\|^2 + t \|v_3\|^2 +  2 \operatorname{Re}\big (\lambda_1 \langle v_2, v_1 \rangle + \lambda_2 \langle v_3, v_1 \rangle + \lambda_3 \langle v_3, v_2 \rangle \big )|\\
&= 4.
\end{align*}
Next, we evaluate the supremum over the torus $z = (z_1, z_2, z_3) \in \mathbb{T}^3$.  Setting $z_j = e^{i\theta_j}$, we obtain:
\begin{align*}
\sup_{z\in \mathbb{T}^3} |\langle B z, z \rangle| &= \sup_{z\in \mathbb{T}^3} | z_1\overline{z_1} - z_2\overline{z_2} + z_3\overline{z_3} + z_1\overline{z_3} + z_3\overline{z_1} | \\
&= \sup_{\theta_1, \theta_3 \in [0, 2\pi)} |1 + 2 \cos (\theta_1-\theta_3)|\\
&= 3.
\end{align*}

Since $|\langle A, B \rangle| = 4 > 3 = \sup_{z\in \mathbb{T}^3} |\langle B z, z \rangle|$, the reduction step in \cite{BM} claiming that it is enough to show $|\langle A, B \rangle| \leq \sup_{z\in \mathbb{T}^3} |\langle B z, z \rangle|$ for all self-adjoint matrices $B$ used in the proof of Fact 7 fails for the rank-1 extreme point $A\in \mathcal{X}$. 

However, $|\langle A, B \rangle| = 4 < \|B\|_{\ell_\infty^3 \to \ell_1^3}$ since the operator norm $\|B\|_{\ell_\infty^3 \to \ell_1^3}$ is evidently at most $5$ and is achieved at $x = (1, -1, 1)^T$ yielding $\|B\|_{\ell_\infty^3 \to \ell_1^3} = 5$. Thus, the ultimate goal of the proof of Fact 7—showing $|\langle A, B \rangle| \leq \|B\|_\epsilon = \|B\|_{\ell_\infty^3 \to \ell_1^3}$ for any extreme points $A$ of $\mathcal{X}$—remains unbroken by this choice of $A$ and $B$. Indeed, Theorem \ref{thm-propQ-linf3} verifies this inequality. 

\section{Proof of Lemma \ref{norm:lem 2.17}}
For any pair of real numbers $a, b$, let $A$ be the $2\times 2$ matrix: 
$\Big ( \begin{matrix} a & re^{i\theta} \\ re^{-i\theta} & b\end{matrix} \Big )$, where $r \geqslant 0$ and $0 \leqslant \theta < 2\pi$. We wish to compute $\Big \| \Big ( \begin{matrix} a & re^{i\theta} \\ re^{-i\theta} & b\end{matrix} \Big )\Big \|_{\infty\to 1}$.  First, note that 
\begin{align*}
\Big \|\Big ( \begin{matrix} a & re^{i\theta} \\ r e^{-i\theta} & b\end{matrix} \Big ) \Big ( \begin{matrix} e^{i \psi}\\e^\varphi \end{matrix} \Big )  \Big\|_1 & = 
\big |a e^{i\psi} + r e^{i(\theta+\varphi)} \big | + \big |r e^{i(- \theta+\psi)} + b e^{i\varphi} \big |\\
&=\big | a + re^{i(\theta + \varphi - \psi)}\big | + \big | r + b e^{i(\theta + \varphi - \psi)}\big |\\
&= \big | a + r \cos x + i \sin x \big |  + \big | r + b \cos x + i \sin x\big | \\
&= \sqrt{a^2 + r^2 + 2 a r \cos x} + \sqrt{r^2 + b^2 + 2 r b \cos x}, 
\end{align*}
where $x=\theta + \varphi - \psi$. The extremal problem we have to solve is 
the following: 
\[\sup_{x\in [0,2\pi)} f(x), \quad \text{\rm where~} f(x) = \sqrt{a^2 + r^2 + 2 a r \cos x} + \sqrt{r^2 + b^2 + 2 r b \cos x}.\]
The critical points of $f$ are the zeros of the derivative 
\[f^\prime(x) = r \sin x \Big (\frac{a}{\sqrt{a^2 + r^2 + 2 a r \cos x}} + \frac{b}{\sqrt{r^2 + b^2 + 2 r b \cos x}} \Big ).\]
Assume that $r\ne 0$, otherwise, $\|A\|_{\infty\to 1} = |a| + |b|$. If $a,b$ are of the same sign, then 
$f^\prime(x) = 0$ if and only if  $\sin x = 0$, i.e., $x=0$ or $x=\pi$. If $a, b$ are of opposite sign, then apart from the two solutions, $x=0, \pi$,  $x$ such that 
\[b\sqrt{a^2 + r^2 + 2 a r \cos x} + b \sqrt{r^2 + b^2 + 2 r b \cos x} = 0\]
is also a critical point of $f$. This critical point occurs when 
\[\cos x = \frac{-r(a+b)}{2 a b}.\]
Clearly, $\|A\|_{\infty\to 1} \leqslant |a| + |b| + 2 r$. 
If $a, b$ are of the same sign, taking $x=0$, we see that $\|A\|_{\infty\to 1} = |a| + |b| + 2 r$. 

If $a, b$ are both not zero and are of opposite signs, then  $\|A\|_{\infty\to 1} < |a| + |b| + 2 r$. So, the critical point $0$ cannot be a point of maximum for $f$. Therefore, in this case, the maximum of $f$ can be either at 
$x=\pi$, or at an $x$ such that $\cos x = \frac{-r(a+b)}{2 a b}$. 

Assume without loss of generality that $a$ is non-negative and $b = - s$, $s\geqslant 0$. Then 
$\cos x=-r(a+b) /(2 a b)=r(a-s) /(2 a s)$. 
In this case: 

\begin{subequations}
    \begin{align}
f(x)_{|x=0} = |s - r| + (r + a) &=\begin{cases}  s + a & \text{if~} s \geq r\\
    2r + a -s &\text{if~} s \leq r.
    \end{cases}\\
f(x)_{|x=\pi} =s + r + |a - r| &= \begin{cases}  s + a & \text{if~} a \geq r\\
    2r + s - a &\text{if~} a \leq r. \end{cases}\\
    f(\cos\,x)_{| \cos \, x = \frac{-r(a+b)}{2 a b}} &= \sqrt{\tfrac{a}{s}(as+ r^2)} + \sqrt{\tfrac{s}{a}(as+ r^2)} \nonumber \\
    &=\sqrt{\frac{as+r^2}{as}} (a+s).
    \end{align}
\end{subequations}
Thus, the maximum possible value of $f$ is one of the following 
\[a+s, 2r+a-s, 2r+s-a, \sqrt{\frac{as+r^2}{as}} (a+s).\]
However, $\sqrt{\tfrac{as+r^2}{as}} (a+s) \geq a + s$. Therefore, $a + s$ is not a maximum of $f$. 

Suppose that $a=s$. Then $\sqrt{\frac{as+r^2}{as}} (a+s) > 2 r$. Therefore, in this case, the maximum of $f$ is $ 2 s \sqrt{\frac{s^2+r^2}{s^2}} $.

Now, suppose that $s\geq r$, then the maximum of $f$ is 
\[\max \{ 2r + s - a = 2(r-s) + (a+s), \sqrt{\tfrac{as+r^2}{as}} (a+s)\}.\]
Therefore, maximum of $f$ is either $2r + s - a$ or $\sqrt{\tfrac{as+r^2}{as}} (a+s)$ according as 
\[2(r - s) \geq \big (\sqrt{1+ \tfrac{1}{as}} - 1\big ) (a+s)\] 
or not. A similar computation applies to the case: $a \geq r$.   

Finally, if $s \leq r$, then there are two cases: either $a \geq s$ or $a \leq s$. Choose, for instance, $a \geq s$. Then $2r + a - s \geq 2r + s - a $. Hence the maximum of $f$ is $\max\{2r + a - s,  \sqrt{\tfrac{as+r^2}{as}} (a+s)\}$. It is either $2r + a - s$, or $\sqrt{\tfrac{as+r^2}{as}} (a+s)$ according as 
\[2(r-s) \geq \big (\sqrt{1+ \tfrac{1}{as}} - 1 \big) (a+s)\]
or not.

We get nothing new when $a \leq r$.
\subsubsection*{\sf Remark:}
If $A$ is of the form $\Big ( \begin{matrix} a & re^{i\theta} \\ re^{-i\theta} & -a\end{matrix} \Big )$, $a>0$, the norm $\big \| \Big ( \begin{matrix} a & re^{i\theta} \\ re^{-i\theta} & -a\end{matrix} \Big )\big \|_{\infty \to 1}$ is given by the formula: 
\[\Big \| \Big ( \begin{matrix} a & re^{i\theta} \\ re^{-i\theta} & -a\end{matrix} \Big )\Big\|_{\infty \to 1} = 2 a \sqrt{\big ( 1+\tfrac{r^2}{a^2} \big )}.\]

{\small\subsubsection*{Acknowledgment} 
The first named author acknowledges the funding received through ANRF ARG grant {\tiny ANRF/ARG/2025/001228/MS}.

The second named author would like to thank DST, Govt. of India for partial financial support in the form of INSPIRE Faculty Fellowship {\tiny DST/INSPIRE/04/2022/001207}.

Gadadhar Misra acknowledges extended conversations with Anthropic's Claude that served as a sounding board in the preparation of this paper. These exchanges were iterative by nature, not always conclusive, and on more than one occasion called for a course correction.  Nonetheless, they helped sharpen the exposition, clarify notation, and refine some of the arguments. The mathematical content, the choice and formulation of results, remain the authors' responsibility. 

The fourth named author acknowledges the funding received through the DST-INSPIRE Faculty Fellowship  
{\tiny DST/INSPIRE/04/2020/001132}, the Prime Minister Early Career Research Grant  {\tiny ANRF/ECRG/2024/000699/PMS}
and the ANRF ARG MATRICS grant {\tiny ANRF/ARGM/2025/000895/MTR}.}

\bibliographystyle{amsplain}

\end{document}